\RequirePackage{fix-cm}
\documentclass[10pt]{article}

\usepackage[T1]{fontenc}
\usepackage[utf8]{inputenc}
\usepackage[british]{babel}
\usepackage{lmodern}
\usepackage{latexsym}
\usepackage{amsmath}
\usepackage{amsthm}
\usepackage{amssymb}
\usepackage{enumerate}
\usepackage{bm}
\usepackage{url}
\usepackage{mleftright}
\usepackage{ifpdf}
\usepackage[svgnames]{xcolor}
\usepackage{microtype}

\newif\ifforceblacklinks\forceblacklinksfalse

\ifpdf
  \PassOptionsToPackage{pdftex,breaklinks}{hyperref}%
\fi%
\usepackage{hyperref}%

\newcommand{\doi}[1]{\href{https://doi.org/#1}{#1}}%

\newcommand{\set}[1]{\ensuremath{\mleft\{ #1 \mright\}}}
\DeclareRobustCommand{\lset}[2]{\ensuremath{\mleft\{\mleft.#1\ \vphantom{#2}\mright|\ #2\mright\}}}
\newcommand{\Ona}[2]{\ensuremath{\mathrm{O}^{(#1)}_{#2}}}
\newcommand{\OA}[1]{\ensuremath{\mathrm{O}_{#1}}}
\newcommand{\Ops}[1][A]{\OA{#1}}
\DeclareMathOperator{\Pol}{Pol}
\DeclareMathOperator{\Aut}{Aut}
\DeclareMathOperator{\End}{End}
\DeclareMathOperator{\Emb}{Emb}

\DeclareMathOperator{\Hom}{Hom}
\DeclareMathOperator{\Sym}{Sym}
\DeclareMathOperator{\CSP}{\text{\textsf{CSP}}}
\DeclareMathOperator{\id}{id}

\DeclareMathOperator{\betw}{betw}
\DeclareMathOperator{\crc}{circ}
\DeclareMathOperator{\sep}{sep}
\newcommand{\Fn}[2][n]{#2^{(#1)}}%
\newcommand{\powerset}[1]{\ensuremath{\mathfrak{P}\apply{#1}}}
\newcommand{\apply}[1]{\ensuremath{\mleft( #1 \mright)}}
\newcommand{\fapply}[1]{\ensuremath{\mleft[ #1 \mright]}}
\newcommand{\gapply}[1]{\ensuremath{\mleft\langle #1 \mright\rangle}}
\newcommand{\subs}{\subseteq}%
\newcommand{\defeq}{\mathrel{\mathop:}=}

\newcommand{\nsg}{\trianglelefteq}%
\newcommand{\exi}{\overline{\xi}}%
\renewcommand{\phi}{\varphi}%
\newcommand{\cl}[1]{\ensuremath{\overline{#1}}}%
\newcommand{\bfa}[1]{{\bm#1}}%
\newcommand{\N}{\mathbb{N}}
\newcommand{\Q}{\mathbb{Q}}
\newcommand{\cpl}[1]{\mathbb{#1}^{\complement}}%
\newcommand{\alt}{\,/\hspace{0pt}\,}%
\newcommand{\inbr}[1]{\textup{(}#1\textup{)}}
\newcommand{\nbd}[1]{{#1}\nobreakdash-\hskip0pt}
\DeclareRobustCommand{\nlb}{\penalty10000\hskip0pt\relax}%
\DeclareRobustCommand{\nbdd}[1]{\nbd{\mbox{$#1$}}}%
\DeclareRobustCommand{\dash}{\nolinebreak\hskip0pt-\hskip0pt}
\let\OLDrestriction\restriction  
\renewcommand{\restriction}{\mathclose\OLDrestriction}
\newcommand{\Restriction}{|}

\allowdisplaybreaks

\hypersetup{
    bookmarks=true,          
    unicode=true,            
    pdftoolbar=true,         
    pdfmenubar=true,         
    pdffitwindow=false,      
    pdfstartview={FitH},     
    pdfdisplaydoctitle=true, 
    pdfnewwindow=true,       
    linktocpage=true,        
    colorlinks=true,         
    linkcolor=DarkSlateBlue, 
    citecolor=Navy,          
    filecolor=DarkViolet,    
    urlcolor=SteelBlue,      
    pdftitle={On a stronger reconstruction notion for monoids and clones},    
    pdfauthor={Mike Behrisch},     
    pdfsubject={Extending Rubin-type reconstruction results
    (forall-exists-interpretations) up to conjugation from
    automorphism groups to monoids and clones},   
    pdfkeywords={automatic action compatibility, automatic homeomorphicity, reconstruction, weak forall-exists-interpretation, transformation monoid, clone, topological monoid, topological clone, homogeneous structure},         
}
\ifforceblacklinks%
\hypersetup{%
    linkcolor=black,     
    citecolor=black,     
    filecolor=black,     
    urlcolor=black       
}%
\fi%

\title{On a stronger reconstruction notion for monoids and clones%
\thanks{Preliminary versions of these results were presented at
the `96.~Ar\-beits\-ta\-gung All\-ge\-mei\-ne Al\-ge\-bra'
(`96th Workshop on General Algebra'), AAA96, that took place in
Darm\-stadt, Germany, 1--3 June 2018,
at the `56th Summer School on Algebra and Ordered Sets',
held in \v{S}pindler\r{u}v Ml\'yn, Czech Republic, 2--7 September 2018,
and at the meeting `From permutation groups to model theory: a workshop
inspired by the interests of Dugald Macpherson, on the occasion of his
60th birthday' at the ICMS, Edinburgh, UK, 17--21 September 2018.}}
\author{Mike Behrisch$^{1,3}$  and Edith Vargas-Garc\'\i a$^{2,4}$}
\date{}
\begin{document}
\maketitle
\newtheorem{theorem}{Theorem}[section]
\newtheorem{corollary}[theorem]{Corollary}
\newtheorem{definition}[theorem]{Definition}
\newtheorem{remark}[theorem]{Remark}
\newtheorem{example}[theorem]{Example}
\newtheorem{problem}{Problem}
\newtheorem{fact}[theorem]{Fact}
\newtheorem{lemma}[theorem]{Lemma}
\newtheorem{proposition}[theorem]{Proposition}

\setcounter{footnote}{1}\footnotetext{In\-sti\-tut f\"ur Dis\-kre\-te
Ma\-the\-ma\-tik und Geo\-me\-trie, Tech\-ni\-sche Uni\-ver\-si\-t\"at Wien, A-1040 Vienna, Austria; e-mail \url{behrisch@logic.at}}
\setcounter{footnote}{2}\footnotetext{Department of Mathematics, ITAM,
R\'\i o Hondo 1, Ciudad de M\'exico 01080, Mexico; e-mail
\url{edith.vargas@itam.mx}}
\setcounter{footnote}{3}\footnotetext{The research of the first author
was partly supported by the OeAD KONTAKT project CZ~02/2019 `Function
algebras and ordered structures related to logic and data fusion'.}
\setcounter{footnote}{4}\footnotetext{The second author gratefully
acknowledges financial support by the Asociaci\'on Mexicana de Cultura A.C.}

\begin{abstract}
Motivated by reconstruction results by Rubin, we introduce a new
reconstruction notion for permutation groups, transformation monoids and
clones, called automatic action compatibility, which entails automatic
homeomorphicity. We further give a characterization of automatic
homeomorphicity for transformation monoids on arbitrary carriers with a
dense group of invertibles having automatic homeomorphicity. We then
show how to lift automatic action compatibility from groups to monoids
and from monoids to clones under fairly weak assumptions. We finally
employ these theorems to get automatic action compatibility results for
monoids and clones over several well-known countable structures,
including the strictly ordered rationals, the directed and undirected
version of the random graph, the random tournament and bipartite graph,
the generic strictly ordered set, and the directed and
undirected versions of the universal homogeneous Henson graphs.
\bgroup
\let\thefootnote\relax%
\footnote{\noindent\emph{2010 Mathematics Subject Classification.}%
  \begin{tabular}[t]{ll}%
  Primary:&
  08A35, 
  08A40, 
  54H15; 
  \\
  Secondary:&
  08A02, 
  03C15  
  03C40  
  \end{tabular}}%
\setcounter{footnote}{4}%
\egroup
\end{abstract}

\section{Introduction}

Permutation groups, transformation monoids and clones carry a natural
top\-ology: the Tichonov topology (also known as the topology of pointwise
convergence). Under this
topology the corresponding group becomes a topological group since composition and taking inverses are continuous operations, similarly
a transformation monoid becomes a topological monoid, and clones (see
below for a definition) become topological clones; here again the composition
becomes continuous with respect to the topology induced by the Tichonov
topology.
\par

When we pass from a (usually countable) structure $\mathbb{A}$ to its
automorphism group $\Aut\apply{\mathbb{A}}$ seen as permutation group, then
to $\Aut\apply{\mathbb{A}}$ seen as a topological group, and finally to
$\Aut\apply{\mathbb{A}}$ as an abstract group, we keep losing
information, and a principal problem is: how much of this information can be recovered?
\par

Hence, some of the natural research questions in this field
relate to `reconstruction' issues.
For automorphism groups of (usually countable) structures, the reconstruction problem can be roughly stated as the question of determining as much
about a structure as possible given only its automorphism group (up to
some notion of equivalence). If one restricts the class of structures
one looks at, there are various famous instances in which a partial or
complete answer can be found. For instance, in~\cite{Rubin1994},
\cite{RubinMcCleary2005} and~\cite{Rubin2016}, M.~Rubin tackles
reconstruction problems for chains, Boolean algebras, certain
\nbdd{\aleph_0}categorical trees, a broad class of
\nbdd{\aleph_0}categorical structures in general, and several other
specific structures.
\par

One popular method, which by no means applies in all cases,
is to establish the `small index property' (SIP) for a structure.
SIP means in its original form that for a
countable structure~$\mathbb{A}$ any
subgroup of its automorphism group of countable index contains the pointwise stabilizer of a finite set.
Reinterpreted via the
natural topology on the automorphism group this is saying that any
subgroup of countable index is open.
Seeking a more robust notion of reconstruction leads to the definition that a countable structure~$\mathbb{A}$ is `reconstructible'
if any isomorphism from its automorphism group to a closed subgroup of
$\Sym(\omega)$ is an isomorphism of the topological groups, that is, a
homeomorphism. While the SIP is mainly useful in the countable case, the
latter concept also makes sense for carriers of arbitrary cardinality.
\par

If we seek to apply the above notions to clones, or even to monoids
naturally associated with~$\mathbb{A}$ such as the self\dash{}embedding
or endomorphism monoid,
there are immediate problems regarding the SIP. The notion of `index' does not arise naturally in the semigroup context.
The better way to view it therefore is in the formulation via `automatic
homeomorphicity'~\cite{Bodirsky}.  Automatic homeomorphicity means that every
(abstract) isomorphism between two (closed) permutation groups, transformation semigroups or
two clones on carriers of the same size is a homeomorphism with respect
to the Tichonov topologies. That is, the isomorphism class of a
particular (closed) permutation group, understood as an abstract group, coincides
with its isomorphism class seen as a topological group, and likewise for
monoids and clones. A naturally arising question is then, of course,
which monoids or clones are `reconstructible', or have automatic
homeomorphicity.
\par

One of Rubin's central reconstruction results~\cite{Rubin1994} gives a
criterion for the class of countable \nbdd{\aleph_0}categorical
relational structures without algebraicity. Whenever such a
structure~$\mathbb{A}$ has a so\dash{}called
weak \nbdd{\forall\exists}interpretation and~$\mathbb{B}$ is countable \nbdd{\aleph_0}categorical without algebraicity, it is
enough to know $\Aut\apply{\mathbb{A}}\cong \Aut\apply{\mathbb{B}}$ in
order to conclude that the permutation groups
$\gapply{\Aut\apply{\mathbb{A}}, A}$ and $\gapply{\Aut\apply{\mathbb{B}}, B}$
are isomorphic. More precisely, if~$\mathbb{A}$ and~$\mathbb{B}$ are
countable \nbdd{\aleph_0}categorical structures without algebraicity,
$\mathbb{A}$ has a weak \nbdd{\forall\exists}interpretation and
$\xi\colon \Aut\apply{\mathbb{A}}\to \Aut\apply{\mathbb{B}}$ is an
isomorphism, then there exists a bijection~$\theta$ between the domain
of~$\mathbb{A}$ and the domain of $\mathbb{B}$ such that for all
$g \in \Aut\apply{\mathbb{A}}$ we have
$\xi(g) = \theta\circ g \circ \theta^{-1}$.
\par

We take this result as a motivation to define a stronger property
analogous to automatic homeomorphicity: we ask that any abstract
algebraic isomorphism from a permutation group, transformation monoid or
clone to another such structure from a given class, having an equipotent
carrier, is a concrete isomorphism, i.e., one that respects the action. So, for a
permutation group, transformation monoid or clone with this `automatic
action compatibility' the algebraic isomorphism class coincides with its
action isomorphism class. Such a property implies automatic
homeomorphicity and gives rise to stronger reconstruction results for
countably infinite first\dash{}order structures, see
section~\ref{subsect:reconstruction}.
\par

For the example structures we
consider, these immediate consequences concerning automatic
homeomorphicity of the monoids of self\dash{}embeddings are special
cases of results known from the literature (\cite[Corollary~22,
p.~3726]{Bodirsky} for the countable universal homogeneous graph,
directed graph, tournament, \nbdd{k}uniform hypergraph ($k\geq 2$),
\cite[Theorem~2.6, p.~75]{BehReconstructingTopologyOnRationals} for
$(\Q,<)$, \cite[Theorems~3.7 and~4.1]{TrussVargasReconstructingTopologyMonoidsPolClonesOfRationals} for the
rationals with the circular order or the separation relation,
respectively, \cite[Theorem~2.2, p.~599]{PechReconstructingTopologyOfEEmbCntblSaturatedStr}
for the Henson graphs and digraphs).
In fact, one can get these results in a uniform way by combining the
arguments in Corollary~\ref{cor:easier-assumptions} and
Lemma~\ref{lem:str-partially-satisfying-assumptions} with Corollary~3.17
of~\cite[p.~11]{PechPolymorphismClonesHomStr}.
However, with the exception of the random graph (see~\cite[Theorem~5.2,
p.~3737]{Bodirsky}), the random directed graph, the \nbdd{k}uniform
hypergraph
(see~\cite[Example~4.6, p.~19 et seq.]{PechPolymorphismClonesHomStr}) and the strictly ordered generic poset
(\cite[Theorem~4.7, p.~20]{PechPolymorphismClonesHomStr}) not much is
known on the side of the polymorphism clones. Hence, all automatic
homeomorphicity corollaries one can derive from our reconstruction
results shed some fresh light on polymorphism clones related with the
studied example structures. In particular, it
follows from Lemma~\ref{lem:str-partially-satisfying-assumptions} that
any polymorphism clone on~$\mathbb{Q}$ having $\End(\Q,<)$ as its
unary part has automatic homeomorphicity w.r.t.\ the class of all
polymorphism clones of countable \nbdd{\aleph_0}categorical structures
without algebraicity. As a special case we therefore
obtain new theorems regarding automatic homeomorphicity of
$\Pol(\Q,<)$ itself and of the polymorphism clones of the reducts of the
strictly ordered rationals, $\Pol(\Q,\betw)$, $\Pol(\Q,\crc)$ and $\Pol(\Q,\sep)$, see
Lemma~\ref{lem:str-partially-satisfying-assumptions} and
Corollaries~\ref{cor:Q-betw} and~\ref{cor:reconstruction-ssip}. A better understanding of these
structures has been one of our initial motivations for studying
automatic action compatibility.
However, the obtained results are conditional in that the scope of the
automatic action compatibility / homeomorphicity condition is
restricted to a proper subset of all closed clones on countable carrier
sets.
Proving an unconditional automatic homeomorphicity result
for the polymorphism clone of~$(\Q,<)$, as well as for the ones of its
reducts, remains an open problem for the time being.
\par

The central contribution of this article is that we provide a machinery how to
transfer automatic action compatibility from the group case to suitable
transformation monoids and clones. In this respect, it turns out that
in concrete instances the essential structural obstructions (i.e.,
additional assumptions that need to be made) come from the
group case; our lifting techniques work fairly generally and have, for
example, no restrictions on the cardinality of the carrier sets. The
main technical condition that we have to impose to cross the border from
groups to monoids is that the group lie dense in the monoid, but this
does not present a severe restriction from the viewpoint of applications
(cf.~\cite[section~3.7, p.~3716]{Bodirsky}). In this setting, as a
by-product, we prove a new characterization of automatic
homeomorphicity for transformation monoids on arbitrary carrier sets. It
involves a weakening of the technical assumption that the only injective
monoid endomorphism fixing every group member be the identical
one, which has featured in several earlier reconstruction
results~\cite{Bodirsky,PechReconstructingTopologyOfEEmbCntblSaturatedStr,TrussVargasReconstructingTopologyMonoidsPolClonesOfRationals,BehReconstructingTopologyOnRationals}.
We believe that our weakened version, so, by our characterization,
automatic homeomorphicity, would also have been sufficient in any of
these previous cases.
\par

Reconstruction questions have quite a long tradition in mathematics, for
they can, in a broader sense, even be traced back to the problem of
classifying (hence understanding) geometries from their symmetry groups
in Felix Klein's `Er\-lan\-ger Pro\-gramm'~\cite{KleinErlangerProgramm,KleinErlangenProgramme}.
Recently, however, such questions have also gained importance in applied
contexts, namely for studying the complexity of so\dash{}called
`constraint satisfaction problems' (CSPs) in theoretical computer science.
These are decision problems regarding the satisfiability of primitive
positive logical formul\ae{} where the constraints are formulated using (usually
relational) structures~$\mathbb{A}$ with finite signature,
so\dash{}called `templates', given as a defining parameter of
the problem, $\CSP(\mathbb{A})$. Besides finite structures, in
particular countably infinite homogeneous structures obtained as
Fra\"\i{}ss\'e{} limits play a key role here because they are
\nbdd{\aleph_0}categorical and can be treated from the point of view of
the Galois connection of polymorphisms and invariant relations like
their finite siblings~\cite[Theorem~5.1]{BodirskyNesetrilConstraintSatisfactionWithCountableHomogeneousTemplates}
Automatic homeomorphicity comes into play when comparing the complexity
of CSPs given by two countable \nbdd{\aleph_0}categorical template
structures $\mathbb{A}$ and~$\mathbb{B}$ of finite but
possibly different relational signature. Namely, if there is a
continuous homomorphism from the polymorphism clone of~$\mathbb{A}$ to
that of~$\mathbb{B}$ whose image is dense (and in contrast to the finite
case, continuity is a true requirement here), then there is a primitive
positive interpretation of~$\mathbb{B}$ into~$\mathbb{A}$
\cite[Theorem~1, p.~2529]{BodirskyPinskerTopologicalBirkhoff}. In this
case one can find a log\dash{}space reduction from
$\CSP(\mathbb{B})$ to $\CSP(\mathbb{A})$, meaning that
$\CSP(\mathbb{B})$ is not essentially harder to solve than
$\CSP(\mathbb{A})$.
Thus the presence of a homomorphism between the clones that is also a
homeomorphism indicates that $\CSP(\mathbb{A})$ and $\CSP(\mathbb{B})$
have the same decision complexity: the problems are interreducible. So
from the viewpoint of CSPs, structures with automatic homeomorphicity
are easier to study as they behave more like finite templates where
continuity of homomorphisms is not an issue to establish reductions. As
automatic action compatibility is a new and more powerful concept than
automatic homeomorphicity, its full effect in the context of CSPs yet remains to be explored.

\section{Preliminaries}\label{sect:preliminaries}

\subsection{Maps, groups, monoids, clones and their homomorphisms}
Subsequently we write $\N=\set{0,1,2,\ldots}$ for the set of natural
numbers, and for sets $A$ and $B$ we denote their inclusion by
$A\subs B$ as opposed to proper set inclusion $A\subset B$. Further, we
use the symbol~$\powerset{A}$ to denote the set of all subsets of~$A$.
The cardinality of a set~$A$ will be written as~$\lvert A\rvert$.
Moreover, if $f\colon A\to B$ is a function, $U\subs A$ and $V\subs B$
are subsets, we write $f\fapply{U}$ for the image of~$U$ under~$f$ and
$f^{-1}\fapply{V}$ for the preimage of $V$ under~$f$. If
$f\fapply{U}\subs V$, we write $f\restriction_U^V$ for the function
$g\colon U\to V$ that is given as restriction of~$f$, i.e., $g(x) =
f(x)$ for $x\in U$. If $V=B$, we use the abbreviation
$f\restriction_U\defeq f\restriction_U^B$. We also define the surjective
function $f\Restriction_U\defeq f\restriction_U^{f\fapply{U}}$ for any
$U\subs A$. In this way, if $g\colon C\to D$ is another function and
$U\subs A\cap C$, then $f\Restriction_U = g\Restriction_U$ is equivalent
to $f(u) = g(u)$ for all $u\in U$, because for~$\Restriction_U$ we do
not have to care about the equality of the co\dash{}domain.

For a set~$A$ and $n\in \N$ we denote by
$\Ona{n}{A}\defeq \lset{f}{f\colon A^n\rightarrow A} = A^{A^n}$ the set
of all \nbdd{n}ary operations on~$A$ and by~$\Ops =
\bigcup_{n\in\N}\Ona{n}{A}$ the set of all finitary operations on~$A$.
If $F\subs \Ops$ is a set of finitary operations on~$A$ and $n\in\N$,
we usually write $\Fn{F}$ for $\Ona{n}{A}\cap F$ and call it the
\emph{\nbdd{n}ary part} of~$F$. This is consistent
since $\Fn{\apply{\Ops}} = \Ona{n}{A}$.
We compose maps from right to left, so if $g\colon A\to B$ and
$f\colon B\to C$, then $f\circ g\colon A\to C$, mapping $a\in A$ to
$f(g(a))$. A subset $S\subs \Ona{1}{A}$ of unary operations on~$A$ that
is closed under composition is a transformation semigroup on~$A$. If~$S$
contains the identity~$\id_A$, it is a transformation monoid.
Semigroup homomorphisms need to be compatible with the binary
operation~$\circ$, and monoid homomorphisms have to additionally
preserve the neutral element. In contrast to this abstract notion a
\emph{transformation semigroup homomorphism} (or \emph{action
homomorphism}) between $S\subs\Ona{1}{A}$ and $T\subs\Ona{1}{B}$
consists of a pair~$(\phi,\theta)$ of maps, where $\phi\colon S\to T$ is
a homomorphism as above and $\theta\colon A\to B$ is a map such that
$\phi(f)\circ\theta  = \theta\circ f$ for every $f\in S$. If~$\theta$ is
a bijection, this is the same as
$\phi(f) = \theta\circ f\circ \theta^{-1}$ and we say that the
homomorphism $\phi\colon S\to T$ is \emph{induced by conjugation
by~$\theta$}. This happens in particular, if $(\phi,\theta)$ is an
\emph{action isomorphism}, i.e., if it has an inverse action homomorphism
$(\phi^{-1},\theta^{-1})$. Moreover, a semigroup homomorphism
$\phi\colon S\to T$ \emph{is induced by conjugation}, if there is some
bijection~$\theta\colon A\to B$ such that $(\phi,\theta)$ is an action
homomorphism.
In this case~$\phi$ is automatically injective and an
isomorphism when restricted to its image~$\phi\fapply{S}$. So if~$S$ is
a transformation monoid (a permutation group),
then~$\phi\fapply{S}$ will be a transformation monoid (a permutation
group, respectively), and the restriction of~$\phi$
to~$S$ and $\phi\fapply{S}$ will be a monoid (group) isomorphism, and
even gives an action isomorphism. Note also that any transformation semigroup,
transformation monoid or permutation group on~$A$ and any
bijection~$\theta\colon A\to B$ give rise to an action isomorphism
$(\phi,\theta)$ by defining $\phi\colon S\to \phi\fapply{S}$ by
$\phi(f) \defeq \theta\circ f\circ\theta^{-1}$ for $f\in S$.
\par
We shall also be concerned with composition structures of functions of
higher arity (motivated by higher-ary symmetries of structures), that
are perhaps less known:
A set of finitary operations~$F\subs\Ops$ on a fixed set~$A$ is a \emph{clone} if it
contains the projections and is closed under composition, i.e., $F$ is a
clone if for all $n\in \N\setminus \set{0}$, the set
$\set{e_i^{(n)}\colon (x_1,\dotsc,x_n)\mapsto x_i\mid i\in \set{1,\ldots,n}}$ is contained in~$F$ and for
every \nbdd{n}ary operation~$f\in\Fn{F}$ and \nbdd{m}ary operations
$g_1,\ldots,g_n$ from~$F$ the (\nbdd{m}ary) composition
$f\circ(g_1,\dotsc,g_n)$, given by
\begin{align*}
f\circ(g_1,\ldots,g_n)(\bfa{x})&=f(g_1(\bfa{x}),\ldots,g_n(\bfa{x}))
\end{align*}
for every $\bfa{x}\in A^m$, also belongs to~$F$.

An (abstract) \emph{clone homomorphism} between clones~$F\subs\Ops$
and~$F'\subs\OA{B}$ is a map $\xi\colon F\to F'$ that respects arities
and is compatible with projections and composition. It is a \emph{clone
isomorphism} if it has an inverse clone homomorphism
$\xi^{-1}\colon F'\to F$ with which it composes to the identity on~$F$
and~$F'$, respectively. Generalizing the situation for transformation
monoids, an \emph{action homomorphism} (or \emph{concrete clone
homomorphism}) is a pair $(\xi,\theta)$ where $\xi\colon F\to F'$ is a
clone homomorphism and $\theta\colon A\to B$ is a map such that for
every $n\in\N$ and all $f\in\Fn{F}$ we have
$\xi(f)\circ\apply{\theta\times\dotsm\times \theta} = \theta\circ f$,
where the product is formed with exactly~$n$ factors of~$\theta$.
Again, if $\theta\colon A\to B$ is a bijection, this means
$\xi(f) = \theta\circ f\circ\apply{\theta^{-1}\times\dotsm\times
\theta^{-1}}$
for all $f\in F$, and in this case, by a slight abuse of terminology, we still say that~$\xi$ is
\emph{induced by conjugation by~$\theta$}. This situation clearly
happens for action isomorphisms, having, by definition, an inverse
action homomorphism~$(\xi^{-1},\theta^{-1})$. The same remarks as above
are true: Any clone homomorphism induced by conjugation (i.e., by some
bijection $\theta\colon A\to B$) is injective and a clone isomorphism
(together with~$\theta$ it even is an action isomorphism) when
restricted to its image. Also any clone and any bijection to some other
carrier set induces an action isomorphism in the natural way.
\par

\subsection{Uniformities and topologies on powers}
Functions on a set~$A$ of any fixed arity are members of a power~$A^I$,
namely where $I=A^n$ and~$n$ is the arity of the function. Such powers
of~$A$ (and their subsets) carry a natural topological and even uniform
structure induced on the subset by the power structure on~$A^I$ when~$A$
is initially understood as discrete topological or uniform space.
\par

We only give a short introduction to uniform spaces. A uniformity
on a set~$A$ is a non\dash{}empty (lattice) filter $\mathcal{U}\subs
\powerset{A\times A}$ of reflexive binary relations on~$A$ that is
closed under taking inverses (that is, with every $\alpha\in\mathcal{U}$ also
$\alpha^{-1} = \lset{(y,x)\in A^2}{(x,y)\in\alpha}\in\mathcal{U}$) and
has the property that for every $\alpha\in \mathcal{U}$ there is some
$\beta\in\mathcal{U}$ such that $\beta\circ \beta = \lset{(x,z)\in
A^2}{\exists y\in A\colon (x,y),(y,z)\in\beta}\subs\nlb\alpha$, which
represents a property analogous to the triangle inequality for metric
spaces. The members $\alpha$ of a uniformity~$\mathcal{U}$ are called
\emph{entourages}, and the idea is that $(x,y)\in\alpha$ means that
points~$x$ and~$y$ of~$A$ are uniformly close to each other.
A uniform space is a pair~$(A,\mathcal{U})$ where $\mathcal{U}$ is a
uniformity on~$A$. A \emph{uniformity base $\mathcal{B}$ on~$A$} is a
filter base of binary relations on~$A$ (called \emph{basic
entourages}) such that the filter generated by~$\mathcal{B}$, that is,
the set of all relations containing some $\beta\in\mathcal{B}$, is a
uniformity on~$A$. This requires that $\mathcal{B}\subs\mathcal{U}$, so
every $\beta\in\mathcal{B}$ must be reflexive. Thus a sufficient
condition for~$\mathcal{B}\subs\powerset{A}$ to be a uniformity base is
that~$\mathcal{B}$ is a non\dash{}empty downward directed collection of
reflexive binary relations that is closed under taking inverses and for
every $\beta\in\mathcal{B}$ there is a $\gamma\in\mathcal{B}$ such that
$\gamma\circ\gamma\subs\beta$.
It is well known that every uniformity~$\mathcal{U}$ on~$A$ induces a
topology by saying that a set $U\subs A$ is open if for every $x\in U$
there is an entourage $\alpha\in\mathcal{U}$ such that
$[x]_{\alpha} = \lset{y\in A}{(x,y)\in\alpha}\subs U$. This means the
collection $\lset{[x]_\alpha}{\alpha\in\mathcal{U}}$ forms a
neighbourhood base of~$x$ for each $x\in A$. A map $h\colon A\to B$
between uniform spaces $(A,\mathcal{U})$ and $(B,\mathcal{V})$ is said
to be \emph{uniformly continuous} if for every $\beta\in\mathcal{V}$ there is
some $\alpha\in\mathcal{U}$ such that
$\apply{h\times h}\fapply{\alpha}\subs \beta$. It is clear that it is
sufficient to require this condition to be satisfied for all
$\beta\in\mathcal{B}$ where~$\mathcal{B}$ is a uniformity base
of~$\mathcal{V}$. Moreover, we say that $h\colon A\to B$ is a
\emph{uniform homeomorphism} if it is a bijection and $h$ and $h^{-1}$
are both uniformly continuous.
Of course uniform continuity implies continuity with
respect to the topologies induced by the uniform spaces, so uniform
homeomorphicity implies homeomorphicity.
\par

For our purposes only two uniformities are relevant: the first is the
discrete uniformity, which is generated by the
base~$\bigl\{\Delta^{(2)}_A\bigr\}$ where $\Delta^{(2)}_A=\set{(x,x)\mid x\in A}$.
Thus the discrete uniformity has every reflexive binary
relation as an entourage, and hence induces the discrete
topology~$\powerset{A}$. The second is the uniformity induced by the
discrete one (on~$A$) on sets $F\subs A^I$ where~$I$ is any index set.
The category of uniform spaces with uniformly continuous maps has
products and they are given by equipping the Cartesian product with the
least uniformity on the product such that all projections are uniformly
continuous. Applied to our situation, this implies that a uniformity
base of $A^I$ is given by all equivalence relations of the form
$\alpha_J = \lset{(f_1,f_2)\in \apply{A^I}^2}{
                                f_1\restriction_J = f_2\restriction_J}$
where~$J$ ranges over all finite subsets $J\subs I$.
This induces a uniformity base on $F\subs A^I$ by restricting the basic
entourages to
$\alpha_J\cap F^2 =
        \lset{(f_1,f_2)\in F^2}{f_1\restriction_J = f_2\restriction_J}$.
We note that the product uniformity on~$A^I$ induces the standard
product topology (Tichonov topology) on~$A^I$.
\par

Based on the product uniformity every permutation group\alt
transformation monoid\alt transformation semigroup $F\subs A^A$ carries a natural uniform
structure related to functions interpolating each other on finite
subsets of their domain. Also every clone $F\subs \Ops$ can be written as
$F = \dot{\bigcup}_{n\in\N}\Fn{F}$ and hence be equipped with the
coproduct uniformity given by the uniform structures on each~$\Fn{F}$,
$n\in \N$. Whenever we will be using concepts like openness, closedness,
topological closure, interior, continuity etc., we will implicitly be
referring to the topology induced by the uniformity just described.
In particular, it makes sense to ask whether homomorphisms between
transformation semigroups or between clones are uniformly continuous.
It follows from the definition of the coproduct and the uniformity
induced on subspaces that a clone homomorphism $\xi\colon F\to F'$ is
uniformly continuous if and only if for every $n\in\N$ the restriction
$\xi\restriction_{\Fn{F}}^{\Fn{F'}}\colon \Fn{F}\to\Fn{F'}$ is uniformly
continuous, and the uniformities involved in this condition are the ones
previously described. For more detailed information on the general
product, coproduct and subspace constructions occurring in the above we
refer the reader to the excellent and concise overview given in
section~2 of~\cite{SchneiderUniformBirkhoff}. A more thorough treatment
of uniform spaces can also be found in chapter~9
of~\cite[p.~238]{WillardGeneralTopology}.
\par

It is worth noting that any clone (group, monoid, semigroup) isomorphism
that is induced by conjugation automatically is a uniform homeomorphism.
This underlines the importance of action isomorphisms in the context of
automatic homeomorphicity.
\par

\subsection{Relational structures}
Our main source for permutation groups, transformation monoids\alt
semi\-groups and for clones shall be sets of homomorphisms of relational
structures. If $\mathbb{A}$ and $\mathbb{B}$ are relational structures
of the same signature on~$A$ and~$B$, respectively, then a map $h\colon
A\to B$ is a homomorphism between~$\mathbb{A}$ and~$\mathbb{B}$ if for
any $m\in\N$ and any \nbdd{m}ary relational symbol of the common
signature that is interpreted in~$\mathbb{A}$ and~$\mathbb{B}$ as
$R\subs A^m$ and $S\subs B^m$, respectively, the following implication
is true: for every $\bfa{x}=(x_1,\dotsc,x_m)\in R$ it is required that
$h\circ \bfa{x} = (h(x_1),\dotsc,h(x_m))\in S$. We usually denote the
truth of this fact by $h\colon \mathbb{A}\to\mathbb{B}$ and collect all
homomorphism between~$\mathbb{A}$ and~$\mathbb{B}$ in the set
$\Hom(\mathbb{A},\mathbb{B})$.
If $\mathbb{A}=\mathbb{B}$ then any homomorphism
$h\colon\mathbb{A}\to\mathbb{B}$ is an \emph{endomorphism
of~$\mathbb{A}$}; the set of all endomorphisms of~$\mathbb{A}$ is
$\End(\mathbb{A})$, it forms a transformation monoid. A homomorphism
$h\colon \mathbb{A}\to\mathbb{B}$ having an inverse homomorphism
$h'\colon \mathbb{B}\to\mathbb{A}$ is an \emph{isomorphism}.
Isomorphisms can be equivalently characterized as those bijections that
not only preserve relations as described above, but also reflect
relations, i.e., for isomorphisms the implication used to define the
homomorphism property is a logical equivalence. An
isomorphism which also is an endomorphism is an \emph{automorphism} of a
structure~$\mathbb{A}$ and all automorphisms of~$\mathbb{A}$ form the
set~$\Aut(\mathbb{A})$, which naturally carries a permutation group
structure. An intermediate concept between homomorphism and
isomorphism is that of an \emph{embedding} that is an injective relation
preserving and reflecting map, i.e., an isomorphism, when restricted to
the image. The set of all self\dash{}embeddings of~$\mathbb{A}$ is
denoted as $\Emb(\mathbb{A})$ and gives us another source of
transformation monoids that are closer to the automorphism group.
Moreover, defining
relations in the product structure component-wise, one can study
homomorphisms between the \nbdd{n}th power~$\mathbb{A}^n$ ($n\in\N$)
and~$\mathbb{A}$, which are called \nbdd{n}ary \emph{polymorphisms
of~$\mathbb{A}$}. The set
$\Pol(\mathbb{A})=\bigcup_{n\in\N}\Hom\apply{\mathbb{A}^n,\mathbb{A}}$
consists of all polymorphisms of~$\mathbb{A}$. Such sets always form
clones, and it is well\dash{}known that the polymorphism clones of
structures on a given set~$A$ are exactly those clones that are closed
in the Tichonov topology.
\par
To obtain some good examples for our results we need to impose some
model theoretic `niceness' properties on infinite structures. One of
these is countable categoricity. We say that a structure~$\mathbb{A}$ is
\emph{\nbdd{\aleph_0}categorical}
if up to isomorphism there is exactly
one model of cardinality~$\aleph_0$ of the first\dash{}order theory
of~$\mathbb{A}$. Thus such~$\mathbb{A}$ cannot be finite.
If~$\mathbb{A}$ itself is countably infinite, then
\nbdd{\aleph_0}categoricity implies that~$\mathbb{A}$ is the only
countable model of its first\dash{}order theory up to isomorphism.
All examples occurring in this paper will be
of the latter form, and, by the Ryll-Nardzewski
Theorem~\cite[Theorem~7.3.1, p.~341]{HodgesModelTheory}, for such
structures countable categoricity can be equivalently formulated as a
condition on~$\Aut(\mathbb{A})$ called oligomorphicity
(cf.~\cite[p.~134]{HodgesModelTheory}). Another property we shall need
is homogeneity. We say that a first\dash{}order structure~$\mathbb{A}$
is \emph{homogeneous} (occasionally called ultra\dash{}homogeneous,
cf.~\cite[p.~325]{HodgesModelTheory}), if any isomorphism between any
two finitely generated substructures of~$\mathbb{A}$ can be extended to
an automorphism of~$\mathbb{A}$. For relational structures finitely
generated substructures coincide with finite substructures, so
homogeneity means that any isomorphism between finite substructures
must be extendable to an automorphism. All example structures
considered in this article have a purely relational signature.
As a third type we shall meet structures~$\mathbb{A}$ \emph{without
algebraicity}, that is, whose automorphism group~$\Aut(\mathbb{A})$ has
no algebraicity. In this context, a permutation group~$G$ on~$A$ has no
algebraicity if the algebraic closure of any (finite) subset~$B$ of~$A$
is trivial, i.e., equal to~$B$
(see~\cite[pp.~134, 330]{HodgesModelTheory}). This means for any
(finite) set~$B\subs A$ the only points having finite orbit with respect
to the pointwise stabilizer~$G_{(B)}$ of~$B$ under~$G$ are those in~$B$.
\par

In order to facilitate the study of concrete examples of
countably infinite homogeneous relational structures, we need the
following easy lemma collecting some basic properties. In this
connection we denote by $\Delta_A^{(m)}$ the \nbdd{m}ary relation
$\set{(x,\dots,x) \mid x\in A}$ on a given carrier set~$A$.
The two final statements of the lemma even hold without the assumption
of homogeneity.
\begin{lemma}\label{lem:hom-structures}
For any homogeneous relational structure
$\mathbb{A} = \apply{A,\apply{R_i}_{i\in I}}$, where $R_i\subs A^{m_i}$
is an \nbdd{m_i}ary relation for each $i\in I$, the following facts
hold.
\begin{enumerate}[\upshape (a)]
\item\label{item:loopless}
      If for every $i\in I$ we have
      $R_i\cap \Delta_A^{(m_i)}=\emptyset$ or
      $R_i\cap\Delta_A^{(m_i)}=\Delta_A^{(m_i)}$, then $\Aut(\mathbb{A})$ is
      transitive.
\item\label{item:cntbl-cat}
      If $I$ is finite and $\lvert A\rvert = \aleph_0$, then
      $\mathbb{A}$ is \nbdd{\aleph_0}categorical.
\item\label{item:emb-dense}
      $\overline{\Aut(\mathbb{A})} = \Emb(\mathbb{A})$
      and the invertible embeddings are precisely the automorphisms.
\item\label{item:emb-cmplmt}
      The structure
      $\cpl{A} =
       \apply{A, \apply{R_i}_{i\in I}, \apply{A^{m_i}\setminus R_i}_{i\in
      I}, A^2\setminus \Delta_A^{(2)}}$
      has the property $\End(\cpl{A})=\Emb(\mathbb{A})$.
\item\label{item:trivial-centre}
      If~$\mathbb{A}$ has no algebraicity, then the centre of
      $\Aut(\mathbb{A})$ contains only the identity.
\end{enumerate}
\end{lemma}
\begin{proof}
\begin{enumerate}[\upshape (a)]
\item By the assumed condition, for every $i\in I$ the \nbdd{i}th
      relation of the induced substructure of~$\mathbb{A}$ on any
      singleton~$\set{a}$ is either empty or contains the constant tuple
      $(a,\dotsc,a)$. This means any two induced singleton substructures
      of~$\mathbb{A}$ are isomorphic, and by homogeneity the unique
      isomorphism between~$\set{a}$ and~$\set{b}$ can be extended to an
      automorphism of~$\mathbb{A}$, whatever $a,b\in A$.
\item This statement is a consequence of the Ryll-Nardzewski Theorem,
      which can be found, for instance, in Corollary~3.1.3
      of~\cite[p.~1607]{MacphersonHomogeneousStructures}.
\item The inclusion
      $\overline{\Aut(\mathbb{A})}\subseteq\Emb(\mathbb{A})$ is generally
      true, the converse follows, because for any $f\in\Emb(\mathbb{A})$
      and any finite subset $B\subseteq A$ the restriction of~$f$ to the
      substructures induced by~$B$ and~$f\fapply{B}$ is an isomorphism,
      which, by homogeneity, can be extended to an automorphism
      of~$\mathbb{A}$.
      The additional statement characterizing the group of invertible
      elements of~$\Emb(\mathbb{A})$ is evident from the definitions.
\item A map $f\colon A\to A$ preserves the complement of a relation, if
      it reflects the relation. In particular, preserving the inequality
      relation is equivalent to injectivity. The statement follows from
      the fact that the embeddings are exactly the injective relation
      preserving and relation reflecting maps.
\item This argument does not require homogeneity of~$\mathbb{A}$.
      If~$f$ belongs to the centre of the automorphism group, its graph
      is invariant for any group member, and so the stabilizer of any
      point~$a\in A$ must fix~$f(a)$ since~$\set{f(a)}$ is (primitive
      positively) definable from~$\set{a}$ and the graph of~$f$. Thus
      the algebraic closure of~$\set{a}$ contains~$f(a)$, however,
      since~$\mathbb{A}$ has no algebraicity, it follows that
      $f(a)=a$ for every~$a\in A$. Hence, the centre of
      $\Aut(\mathbb{A})$ is the singleton~$\set{\id_A}$.
      \qedhere
\end{enumerate}
\end{proof}

Concerning the reconstruction of countable \nbdd{\aleph_0}categorical
structures from their automorphism groups we rely on the notions of
\emph{(strong) small index property} and
\emph{weak \nbdd{\forall\exists}interpretation}. These are actually
properties of the automorphism group considered as a permutation group,
so let $G\subs \Sym(A)$ be a permutation group on a countable set~$A$.
We say that~$G$ has the \emph{small index property} (SIP),
see~\cite[p.~144]{HodgesModelTheory}, if every
subgroup $U\leq G$ of countable index contains the pointwise stabilizer
$G_{(B)}$ of a finite subset~$B\subs A$. This is equivalent to asking
that every subgroup $U\leq G$ with $\lvert G/U\rvert\leq \aleph_0$ is
open in the Tichonov topology on~$A^A$. Strengthening this requirement,
the permutation group~$G$ has the \emph{strong small index property}
(SSIP), see~\cite[p.~146]{HodgesModelTheory}, if for every $U\leq G$ of
countable index in~$G$ there is a finite set $B\subs A$ such that
$G_{(B)}\subs U \subs G_{B}$, that is, $U$ lies between the pointwise
and the setwise stabilizer of~$B$. The notion of \emph{weak \nbdd{\forall\exists}interpretation} needs a
lengthy and technical definition that is nowhere needed in this text,
so we leave it as an undefined black box tool and refer the reader
to~\cite{Rubin1994} or~\cite[p.~1620]{MacphersonHomogeneousStructures}
for more information.
Only the consequences of this property in the context of countable
\nbdd{\aleph_0}categorical structures without algebraicity (as stated in
Theorem~\ref{thm:Rubin-group-categoricity}) are relevant for us.
For brevity we say that a structure~$\mathbb{A}$ has the SIP, SSIP
or a weak \nbdd{\forall\exists}interpretation whenever
$\Aut(\mathbb{A})$ has it.

\subsection{Reconstruction notions}\label{subsect:reconstruction}
For applications to such reconstruction questions that go beyond automorphism
groups we need the notion of \emph{automatic homeomorphicity} and, in
close analogy to the latter, we introduce the concept of \emph{automatic
action compatibility}. Both can be defined for closed permutation
groups, transformation monoids and clones, and in both
definitions we allow some class~$\mathcal{K}$ of permutation groups,
transformation monoids or clones to act as a parameter
restricting the scope of the condition. For transformation monoids and
clones on countable sets the concept of automatic homeomorphicity
originated in~\cite[Definition~6, p.~3714]{Bodirsky}; to the best of our knowledge, the
modification of automatic homeomorphicity relative to a parameter class was first given
in~\cite[Definition~4.1, p.~142]{PechAutHomeoTransfMon}. Our definition
extends both as it allows for any carrier set and a scope
parameter~$\mathcal{K}$.

\begin{definition}\label{def:aut-homeo}
Let $\set{F}\cup\mathcal{K}$ be a class of permutation groups\alt
transformation monoids\alt clones and let~$A$ be the carrier set of~$F$.
We say that~$F$ has
\begin{enumerate}[\upshape (a)]
\item
  \emph{automatic homeomorphicity with respect
  to~$\mathcal{K}$}, if for every
  $F'\in \mathcal{K}$ on a set~$B$ of the same cardinality as~$A$
  any group\alt monoid\alt clone isomorphism $\phi\colon F\to F'$
  is automatically a homeomorphism;
\item
  \emph{automatic action compatibility with respect
  to~$\mathcal{K}$}, if for every
  $F'\in \mathcal{K}$ on a set~$B$ of the same cardinality as~$A$
  any group\alt monoid\alt clone isomorphism $\phi\colon F\to F'$
  is automatically part of an action isomorphism, that is, if there is
  some bijection $\theta\colon A\to B$ such that $\phi$ is induced by
  conjugation by~$\theta$.
\end{enumerate}
It is customary to agree that~$F$ having automatic homeomorphicity
(action compatibility) without any restriction means~$F$ having this
property with respect the class of all (closed, if $F$ is closed)
permutation groups\alt transformation monoids\alt clones on sets
equipotent to~$A$.
\par
To shorten formulations we stipulate that if~$\mathcal{C}$ is a class of
structures, the statement that~$F$ has one of the above properties
`with respect to~$\mathcal{C}$' means that the permutation
group\alt transformation monoid\alt clone~$F$ has the property with
respect to the class~$\mathcal{K}$ of all automorphism groups\alt
endomorphism monoids\alt poly\-morph\-ism clones of structures
in~$\mathcal{C}$.
\end{definition}

Because every isomorphism induced by conjugation is automatically an
(even uniform) homeomorphism, it is immediate that automatic
action compatibility with respect to~$\mathcal{K}$ implies automatic
homeomorphicity with respect to~$\mathcal{K}$.
\par

Often, automatic homeomorphicity is only considered for such~$F$
that are closed in the Tichonov topology. If in this case $\mathcal{K}$
contains some~$F'\cong F$ on some set equipotent to the carrier of~$F$,
and $F'$ is not closed, then, trivially, $F$ has neither
automatic homeomorphicity nor automatic action compatibility with
respect to~$\mathcal{K}$ (because homeomorphisms preserve closedness).
As such parametrizations do not give very interesting notions, one often
restricts the definition to the case where $\set{F}\cup \mathcal{K}$
consists only of closed groups\alt monoids\alt clones.
Definition~\ref{def:aut-homeo} as given above, however, also allows for
some possibly non\dash{}closed~$F$, where non\dash{}closed
$F'\in\mathcal{K}$ would possibly make sense.
So as a rule of thumb, if~$F$ is closed, we normally only consider
classes~$\mathcal{K}$ of closed sets of functions; if it is not, we may
use any~$\mathcal{K}$ in Definition~\ref{def:aut-homeo}.
\par

It is clear from the definition that whether some~$F\in\mathcal{K}$ has
automatic homeomorphicity\alt action compatibility with respect
to~$\mathcal{K}$ only depends on those members of~$\mathcal{K}$ that
live on carrier sets equipotent to the carrier of~$F$, and thus one
could in principle restrict the definition to the case where all members
of~$\set{F}\cup\mathcal{K}$ have equipotent carriers. In fact, if
$F'\in\mathcal{K}$ has carrier~$B$ and $\theta\colon B\to\nlb A$ is a
bijection with the carrier~$A$ of~$F$, then~$F$ has automatic
homeomorphicity\alt action compatibility with respect to~$\mathcal{K}$
if and only if this holds with respect
to~$(\mathcal{K}\setminus\set{F'})\cup\set{F''}$, where $F''$ on~$A$ is
(uniform homeomorphically) action isomorphic to~$F'$ by conjugation
with~$\theta$. Similarly, all other members of~$\mathcal{K}$ that do not
live on~$A$ could be replaced by an action isomorphic copy on~$A$, and
one could do with classes $\set{F}\cup\mathcal{K}$ all of whose members
live on the same set (and then `on a set~$B$ of the same cardinality
as~$A$' could be dropped from the definition). We prefer the version
above because it is more convenient to say `with respect to all
self-embeddings monoids of relational structures without algebraicity'
than `with respect to all self-embeddings monoids of relational
structures without algebraicity on carriers of size $\aleph_{42}$'.
\par

One should note that the concepts given in
Definition~\ref{def:aut-homeo} give rise to reconstruction results in
the \nbdd{\aleph_0}categorical setting. We explain this in the case of
automorphism groups of relational structures; for polymorphism clones
the situation is essentially the same---effectively in the following
results `first\dash{}order' has to be replaced by `primitive
positive'---but the details are a bit more technical,
see~\cite[section~2.1, p.~3708 et seq.]{Bodirsky} for references and
further information and~\cite[p.~3714]{Bodirsky} for remarks on the
monoid case.
Given countable \nbdd{\aleph_0}categorical structures~$\mathbb{A}$
and~$\mathbb{B}$ where $\Aut(\mathbb{A})$ has automatic homeomorphicity
with respect to the class of (automorphism groups of) countable
\nbdd{\aleph_0}categorical structures, the condition
$\Aut(\mathbb{A})\cong \Aut(\mathbb{B})$ implies that these permutation
groups are isomorphic as topological groups, and thus by a theorem of
Coquand (presented in~\cite[Corollary~1.4(ii), p.~67]{%
AhlbrandtZieglerQuasiFinitelyAxiomatizableTotallyCategoricalTheories}),
the structures~$\mathbb{A}$ and~$\mathbb{B}$ are first\dash{}order
bi\dash{}interpretable.
If on the other hand $\Aut(\mathbb{A})$ has automatic action
compatibility with respect to countable \nbdd{\aleph_0}categorical
structures, then the same condition on the automorphism groups implies
that these are even isomorphic as permutation groups, whence as a
consequence of Ryll-Nardzewski's Theorem, $\mathbb{A}$
and~$\mathbb{B}$ are even first\dash{}order interdefinable (see
also~\cite[Proposition~1.3, p.~226]{Rubin1994}).
\par

Thus, automatic action compatibility entails a stronger reconstruction
notion (up to first\dash{}order bidefinability from automorphism groups,
up to primitive positive bidefinability from polymorphism
clones~\cite[cf.~Theorem~5.1, p.~365]{%
BodirskyNesetrilConstraintSatisfactionWithCountableHomogeneousTemplates}) than automatic homeomorphicity.
It is hence an important question, whether there are any good examples
of countable \nbdd{\aleph_0}categorical relational structures having
automatic action compatibility. Fortunately, there are two theorems that
connect our concept to other well\dash{}studied properties, namely to
Rubin's weak \nbdd{\forall\exists}interpretations and to the strong
small index property, which have been established for many known
structures. The relevant results are the following:

\begin{theorem}[{\cite[{Theorem~2.2, p.~228}]{Rubin1994}}]%
                                   \label{thm:Rubin-group-categoricity}
If $G=\Aut(\mathbb{A})$ and $G'=\Aut(\mathbb{B})$ are automorphism
groups of countable \nbdd{\aleph_0}categorical structures without
algebraicity and \inbr{the automorphism group of} $\mathbb{A}$ has a weak
\nbdd{\forall\exists}interpretation, then every group isomorphism
$\phi\colon G\to G'$ is induced by conjugation.
\end{theorem}

\begin{theorem}[{\cite[Corollary~2]{PaoliniShelahReconstructionWithStrongSIP}}]\label{thm:reconstruction-ssip}
If $G=\Aut(\mathbb{A})$ and $G'=\Aut(\mathbb{B})$ are automorphism
groups of countable \nbdd{\aleph_0}categorical structures without
algebraicity such that~$G$ and~$G'$ have the strong small index
property, then every group isomorphism $\phi\colon G\to G'$ is induced
by conjugation.
\end{theorem}

Reformulated in terms of Definition~\ref{def:aut-homeo}, Rubin's
Theorem~\ref{thm:Rubin-group-categoricity} says that the automorphism
group of a countable \nbdd{\aleph_0}categorical structure
with a weak \nbdd{\forall\exists}interpretation and no algebraicity
has automatic action compatibility with respect to all countable
\nbdd{\aleph_0}categorical structures without algebraicity.
The more recent Theorem~\ref{thm:reconstruction-ssip} by Paolini and
Shelah states that the automorphism group of a countable
\nbdd{\aleph_0}categorical structure having SSIP and no algebraicity has
automatic action compatibility with respect to the class of exactly
these structures.
\par

It is worth noting that the assumption of absence of algebraicity
cannot easily be abandoned in
Theorems~\ref{thm:Rubin-group-categoricity}
and~\ref{thm:reconstruction-ssip}. Otherwise one could take some
countable \nbdd{\aleph_0}categorical~$\mathbb{A}$ with countable
signature, SSIP and transitive
automorphism group, e.g.\ $\mathbb{A}=(\mathbb{Q},<)$, and construct an
extension~$\mathbb{B}$ by adding a new element~$b$ to the carrier
and~$\set{b}$ as a unary relation so that~$b$ becomes
a fixed point of $\Aut(\mathbb{B})$. This construction preserves
\nbdd{\aleph_0}categoricity and SSIP, and although $\Aut(\mathbb{A})$
and $\Aut(\mathbb{B})$ would be isomorphic, they would not be action
isomorphic for $\Aut(\mathbb{A})$ is transitive, but $\Aut(\mathbb{B})$
fixes~$b$.
\par

The two theorems above will be our starting point to lift automatic action
compatibility from permutation groups to monoids and finally to
polymorphism clones of structures with certain additional properties in
section~\ref{sect:monoids}. Concrete examples where our approach applies
are given in section~\ref{sect:examples}.
\par

\section{Characterizing automatic homeomorphicity}
In the following we are going to give a characterization of automatic
homeomorphicity for transformation monoids in the situation when the
group of invertible elements lies dense. The importance of this
particular case for applications has been outlined
in~\cite[section~3.7, p.~3716]{Bodirsky}.
For this we develop an improved version of~\cite[Proposition~11,
p.~3720]{Bodirsky}, which does not require the assumption of a countable
carrier set. We thereby also eliminate any metric reasoning used to
obtain it, which ultimately seems to be an artifice going back to an idea
proposed by Lascar in~\cite[p.~31]{Lascar91}. Even though the product
topology (Tichonov topology) on~$A^A$ over a discrete space~$A$ is only
metrizable if~$A$ is countable, a uniform argument can achieve the same
conclusion without restriction to the countable case. In particular such
a result can be obtained just applying interpolation on finite sets
without mentioning completions of metric spaces via equivalence
classes of Cauchy sequences.

The first step is a non-metric variant
of~\cite[Lemma~4.2, p.~142]{PechAutHomeoTransfMon}.
\begin{lemma}\label{lem:cont-impl-unif}
Let $M\subs \Ona{1}{A}$ and $M'\subs \Ona{1}{B}$ be transformation
monoids such that~$M$ has a dense subset~$G$ of invertibles and let
$\xi\colon M\to M'$ be a continuous monoid homomorphism. Then~$\xi$ is
uniformly continuous.
\end{lemma}
\begin{proof}
Consider any finite subset $D\subseteq B$ determining the basic entourage
$\lset{(h_1,h_2)\in M'^2}{h_1\restriction_D = h_2\restriction_D}$.
The set $\lset{h\in M'}{h\restriction_D = \id_B\restriction_D}$ is a
basic open neighbourhood of~$\id_B$. For~$\xi$ is continuous, the
preimage is an open neighbourhood of~$\id_A$, and thus contains a basic
open neighbourhood of~$\id_A$ given by some finite subset $C\subseteq A$:
\[\lset{g\in M}{g\restriction_C=\id_A\restriction_C}
\subseteq
\lset{g\in M}{\xi(g)\restriction_D = \id_B\restriction_D}.\]
We claim
\[
(\xi\times\xi)\fapply{\lset{(f_1,f_2)\in M^2}{f_1\restriction_C =
f_2\restriction_C}}
\subseteq
\lset{(h_1,h_2)\in M'^2}{h_1\restriction_D = h_2\restriction_D},
\]
proving uniform continuity. Namely, if $f_1,f_2\in M$ satisfy
$f_1\restriction_C=f_2\restriction_C$, then by density of~$G$ the basic
open neighbourhood $\lset{f\in M}{f\restriction_C=f_1\restriction_C}$
of~$f_1$ contains some invertible~$g\in G$. Hence $g\restriction_C =
f_1\restriction_C = f_2\restriction_C$, wherefore we have
$(g^{-1}\circ f_i )\restriction_C = \id_A\restriction_C$ for both
$i\in\set{1,2}$. By the above, we infer
\[(\xi(g)^{-1}\circ \xi( f_i))\restriction_D
  =\xi(g^{-1}\circ f_i)\restriction_D = \id_B\restriction_D,\]
or $\xi(g)\restriction_D = \xi(f_i)\restriction_D$ for $i\in\set{1,2}$.
Hence $\xi(f_1)\restriction_D = \xi(f_2)\restriction_D$.
\end{proof}

Now we give the previously advertised strengthened version
of~\cite[Proposition~11]{Bodirsky} having no restrictions on the carrier
sets.
\begin{proposition}\label{prop:extending-dense-groups}
Let $M\subs \Ona{1}{A}$ and $M'\subs \Ona{1}{B}$ be
transformation monoids having subsets $G\subseteq M$ and
$G'\subseteq M'$ of invertibles such that $M\subs\cl{G}$ and $M'$ is
closed. For any continuous monoid homomorphism $\xi\colon G\to G'$
there is a uniformly continuous monoid
homomorphism $\exi\colon M\to M'$ extending~$\xi$. If~$M$ and~$M'$
are closed, $G$ and~$G'$ are dense and~$\xi$ is a homeomorphic
isomorphism, then $\exi$ is a uniform homeomorphism and a monoid
isomorphism.
\end{proposition}
\begin{proof}
To define $\exi\colon M\to M'$ let $f\in M$ and $b\in B$ be given.
Let~$D\subs B$ be any finite set such that $b\in D$. For convenience, it
can always be chosen as~$\set{b}$, but we are not going to assume this.
The set $\beta_D\defeq \lset{(h_1,h_2)\in G'^2}{h_1\restriction_D =
h_2\restriction_D}$ is a basic entourage, given by~$D$.
By Lemma~\ref{lem:cont-impl-unif} we know that~$\xi$ is uniformly
continuous, hence there is an entourage~$\alpha$ such that
$(\xi\times\xi)\fapply{\alpha}\subseteq \beta_D$. Since every entourage
contains a basic entourage there is some finite set $C\subs A$ and  a basic entourage
$\alpha_C=\lset{(g_1,g_2)\in G^2}{g_1\restriction_C =
g_2\restriction_C}\subs\alpha$, which is hence mapped to~$\beta_D$.
Let $\tilde{C}\supseteq C$ be any possibly larger finite
subset of~$A$. As~$G$ is dense in~$M$, there is some group
element~$g\in G$ such
that~$g\restriction_{\tilde{C}} = f\restriction_{\tilde{C}}$. We
put $\exi(f)\restriction_D\defeq \xi(g)\restriction_D$.
Since $b\in D$, we have thus defined $\exi(f)(b)$.
Careful inspection shows that this is indeed well defined:
namely, if for $i\in\set{1,2}$ we have finite subsets $D_i\subs B$ with
$b\in D_i$, entourages $\alpha_i$ with
$(\xi\times\xi)\fapply{\alpha_i}\subs\beta_{D_i}$
and finite subsets $C_i\subs\tilde{C}_i$ of~$A$ with
$\alpha_{C_i}\subs\alpha_i$ and invertibles $g_i\in G$ with
$f\restriction_{\tilde{C}_i} = g_i\restriction_{\tilde{C}_i}$, then for
$D=D_1\cap D_2$ we need to verify that $\xi(g_1)\restriction_D =
\xi(g_2)\restriction_D$. We put $C\defeq C_1\cup C_2$ and, by
density of~$G$ in~$M$, we pick some~$g\in G$ such that
$f\restriction_{C} = g\restriction_{C}$.
Since $C_i\subs\tilde{C}_i$, it follows that
$g_i\restriction_{C_i} =f\restriction_{C_i}$. Likewise, we get
$g\restriction_{C_i}=f\restriction_{C_i}$ because $C_i\subs C$.
Consequently, $g_i\restriction_{C_i} = g\restriction_{C_i}$, so
$(g_i,g)\in \alpha_{C_i}\subs\alpha_i$ and thus,
$(\xi(g_i),\xi(g))\in (\xi\times\xi)\fapply{\alpha_i}\subs
\beta_{D_i}\subs\beta_D$.
This means $\xi(g_i)\restriction_D = \xi(g)\restriction_D$, which is
independent of the index~$i$.
Hence we have $\xi(g_1)\restriction_D = \xi(g_2)\restriction_D$, so that
$\xi(g_1)$ and $\xi(g_2)$ really give non\dash{}contradictory values for
$\exi(f)$ on~$D$. Note also that, by definition, on any finite~$D\subs B$,
the function~$\exi(f)$ coincides with~$\xi(g)$ for some~$g\in G$, so with
some member of~$G'$. Therefore, $\exi(f)\in \cl{G'}\subs\cl{M'}= M'$.
\par
If $g\in G$ and $b\in B$, then taking $D = \set{b}$ and some finite
subset $C\subs A$ with
$(\xi\times\xi)\fapply{\alpha_C}\subs\beta_{\set{b}}$ we are free to
choose~$g\in G$ as its own interpolant on~$C$, to define
$\exi(g)\restriction_{\set{b}}$ via $\xi(g)\restriction_{\set{b}}$. This
means $\exi(g) = \xi(g)$ for any $g\in G$, so~$\exi$ extends~$\xi$.
\par
To prove that~$\exi\colon M\to M'$ is a homomorphism we consider
$f_1,f_2\in M$ and any $b\in B$. We want to show that
$\exi(f_2\circ f_1)(b) = \exi(f_2)(b')$ where $b'\defeq \exi(f_1)(b)$.
According to the definition, we choose finite subsets
$C_b, C_{b'}\subs A$ such that
$(\xi\times\xi)\fapply{\alpha_{C_b}}\subs\beta_{\set{b}}$ and
$(\xi\times\xi)\fapply{\alpha_{C_{b'}}}\subs\beta_{\set{b'}}$. Moreover,
for $\tilde{C}_{b'}=C_{b'}\cup f_1\fapply{C_b}$ we take any $g_1,g_2\in
G$ such that $g_1\restriction_{C_b} = f_1\restriction_{C_b}$ and
$g_2\restriction_{\tilde{C}_{b'}} = f_2\restriction_{\tilde{C}_{b'}}$. It
follows that
\begin{align*}
\exi(f_2)(b')&= \xi(g_2)(b') = \xi(g_2)(\exi(f_1)(b))
= \xi(g_2)(\xi(g_1)(b))\\
&= \xi(g_2)\circ\xi(g_1)(b)= \xi(g_2\circ g_1)(b)  = \exi(f_2\circ f_1)(b)
\end{align*}
since for every $x\in C_b$ we have $f_2(f_1(x))=g_2(f_1(x))= g_2(g_1(x))$
by the construction of $\tilde{C}_{b'}$ as a superset
of~$f_1\fapply{C_b}$.
\par
To see that~$\exi$ is uniformly continuous, consider any finite
set~$D\subs B$ and the basic entourage~$\beta_D$ given by it. According
to the definition of~$\exi$ we choose a finite subset~$C\subs A$ such
that $(\xi\times\xi)\fapply{\alpha_C}\subs\beta_D$. Now if $f_1,f_2\in M$
fulfil $f_1\restriction_C = f_2\restriction_C$, then we can use any
$g\in G$ satisfying
$f_1\restriction_C = g\restriction_C=f_2\restriction_C$ to define
$\exi(f_1)$ and $\exi(f_2)$ on~$D$,
videlicet $\exi(f_1)(b) = \xi(g)(b) = \exi(f_2)(b)$ for every $b\in D$.
This proves $\exi(f_1)\restriction_D = \exi(f_2)\restriction_D$ and hence
uniform continuity.
\par
Finally, we discuss the situation where~$\xi$ is a bijection with
inverse~$\xi^{-1}$. Putting $f=\overline{\xi^{-1}}(f')$ we  have to check
that $\exi\apply{f} = f'$ for every~$f'\in M'$. Consider any~$b\in B$ and
pick a finite subset $C_b\subs A$ such that
$(\xi\times\xi)\fapply{\alpha_{C_b}}\subs\beta_{\set{b}}$. The value
$\exi(f)(b)$ is determined by any \nbdd{G}interpolant for
$f=\overline{\xi^{-1}}(f')$ on~$C_b$; we need to choose this
interpolant in a special way. We find a finite subset \mbox{$D\subs B$}
such that for
$g_1',g_2'\in G'$ with $g_1'\restriction_D = g_2'\restriction_D$ we have
$\xi^{-1}(g_1')\restriction_{C_b} = \xi^{-1}(g_2')\restriction_{C_b}$.
We put $\tilde{D}= D\cup\set{b}$ and we take any~$g'\in G'$ satisfying
$g'\restriction_{\tilde{D}} = f'\restriction_{\tilde{D}}$. By the
definition of~$\overline{\xi^{-1}}$, we obtain
$f\restriction_{C_b} = \overline{\xi^{-1}}(f')\restriction_{C_b}=
\xi^{-1}(g')\restriction_{C_b}$, so $\xi^{-1}(g')\in G$ is a suitable
interpolant for~$f$ on~$C_b$. Using the definition of~$\exi(f)$, we can
infer $\exi(f)(b) = \xi(\xi^{-1}(g'))(b) = g'(b) = f'(b)$, where the last
equality holds since $b\in \tilde{D}$. This shows
$\exi(\overline{\xi^{-1}}(f')) (b) =f'(b)$, and as we have symmetric
assumptions, we can obtain $\overline{\xi^{-1}}(\exi(f)) = f$ for
$f\in M$ by a dual argument.
\end{proof}

Note that to understand how the extension~$\exi$ works, it is sufficient to
fix for each $b\in B$ one finite set $C_b\subs A$ such that
$(\xi\times\xi)\fapply{\alpha_{C_b}}\subs \beta_{\set{b}}$ (using the
continuity of~$\xi$ at~$\id_A$, see Lemma~\ref{lem:cont-impl-unif}).
This information can be `precomputed'. To see what $\exi(f)$ for
some~$f\in M$ does at~$b\in B$, one then simply has to find an
interpolant~$g\in G$ of~$f$ on~$C_b$ and to observe, how~$\xi(g)$
acts on~$b$.
\par

Next we prove the mentioned  characterization of automatic homeomorphicity, which is
closely related to the sufficient condition given
in~\cite[Lemma~12, p.~3720]{Bodirsky}. Our criterion is again independent
of the size of the underlying set and we shall see afterwards how to
derive the analogue of Lemma~12 of~\cite{Bodirsky} from it.
\begin{proposition}\label{prop:char-aut-homeo-dense-invertibles}
Let $\mathcal{K}$ be a class of closed monoids and
$\mathcal{G}=\set{G(M) \mid M\in\mathcal{K}}$ be the corresponding class
of groups~$G(M)$ of invertibles of monoids~$M\in \mathcal{K}$. Moreover,
we assume that $\lset{\cl{G}}{G\in\mathcal{G}}\subs\mathcal{K}$, where
$\cl{G}$ denotes the closure in the full transformation monoid.
Supposing that $M\subs \Ona{1}{A}$ is a closed monoid, its group
$G = G(M)$ of invertibles is dense in~$M$ and has automatic
homeomorphicity w.r.t.~$\mathcal{G}$, the following facts are
equivalent.
\begin{enumerate}[\upshape (a)]
\item\label{item:aut-unif-homeo}
      For any set~$B$ of the same cardinality as~$A$ and any
      monoid $M'\subs\Ona{1}{B}$ satisfying $M'\in \mathcal{K}$ and every
      monoid isomorphism $\theta\colon M\to M'$ it follows that
      $\theta\fapply{G}$ is dense in~$M'$ and~$\theta$ is a uniform
      homeomorphism.
\item\label{item:aut-homeo}
      $M$ has automatic homeomorphicity with respect to~$\mathcal{K}$.
\item\label{item:weak-aut-homeo}
      For any set~$B$ of the same cardinality as~$A$ and any closed
      monoid $M'\subs\Ona{1}{B}$ satisfying $M'\in \mathcal{K}$ and every
      monoid isomorphism $\theta\colon M\to M'$ it follows that~$\theta$
      is continuous.
\item\label{item:grp-extension}
      For any set~$B$ of the same cardinality as~$A$ and
      transformation monoids $M_1\subs M_2\subs\Ona{1}{B}$ such that
      $M_1, M_2\in\mathcal{K}$ and for any monoid isomorphisms
      $\phi_i\colon M\to M_i$ for $i\in\set{1,2}$ the following
      implication holds:
      \[ \phi_1\Restriction_G = \phi_2\Restriction_G \implies
         M_1= M_2 \land \phi_1=\phi_2.\]
\end{enumerate}
\end{proposition}
Note that every $M\in\mathcal{K}$ can be understood as an endomorphism
monoid of some relational structure~$\mathbb{A}$; $G(M)$ is then the
automorphism group of that structure and thus a closed permutation
group. So the requirement that~$G$ have automatic homeomorphicity
w.r.t.~$\mathcal{G}$ is reasonable. Moreover, $\cl{G(M)}$ is always a
submonoid of the monoid of self\dash{}embeddings of~$\mathbb{A}$, and
if~$\mathbb{A}$ is homogeneous, both monoids are actually equal (see
Lemma~\ref{lem:hom-structures}\eqref{item:emb-dense}). Thus,
if~$\mathcal{K}$ is the collection of all endomorphism monoids of a
given class of homogeneous relational structures, then the condition
that $\lset{\cl{G}}{G\in\mathcal{G}}\subs\mathcal{K}$ means that for
each of these structures, $\mathcal{K}$ also contains the corresponding
monoid of self\dash{}embeddings.
Furthermore, if $\mathcal{K}$ consists of the endomorphism
monoids of all possible relational structures on sets of a certain size,
i.e., up to isomorphism, $\mathcal{K}$ is the collection of all closed
transformation monoids on a fixed set, then the assumption
$\lset{\cl{G}}{G\in\mathcal{G}}\subs\mathcal{K}$ is always satisfied
(cf.\ Lemma~\ref{lem:hom-structures}\eqref{item:emb-cmplmt}).
\begin{proof}[Proof of
Proposition~\ref{prop:char-aut-homeo-dense-invertibles}]
As `$\eqref{item:aut-unif-homeo}\Rightarrow
\eqref{item:aut-homeo}\Rightarrow\eqref{item:weak-aut-homeo}$' is
evident we start with the implication
`$\eqref{item:weak-aut-homeo}\Rightarrow \eqref{item:grp-extension}$'.
Assume condition~\eqref{item:weak-aut-homeo} and let
$M_1\subs M_2\subs\Ona{1}{B}$ be transformation monoids on a
set~$B$ of the same cardinality as~$A$ such that $M_1,
M_2\in\mathcal{K}$, so $M_1, M_2$ are closed.
Suppose we have monoid isomorphisms $\phi_i\colon M\to M_i$ that coincide
on~$G$. It follows that~$\phi_i$ is continuous. Now consider any~$f\in M$
and any~$b\in B$. The set $\lset{h'\in M_i}{h'(b)=\phi_i(f)(b)}$ is a
basic open neighbourhood of~$\phi_i(f)$, so, as~$\phi_i$ is continuous,
its preimage $\lset{h\in M}{\phi_i(h)(b)=\phi_i(f)(b)}$ is an open
neighbourhood of~$f$ and hence contains a basic open neighbourhood
of~$f$. Thus, there is a finite set $X_i\subs A$ such that for any finite
$Y\subs A$ with $X_i\subs Y$ we have
\[f\in V_Y = \lset{h\in M}{h\restriction_Y=f\restriction_Y}\subs
\lset{h\in M}{\phi_i(h)(b) = \phi_i(f)(b)}.\]
Now letting $Y=X_1\cup X_2$ we have this inclusion for both
$i\in\set{1,2}$. By density of~$G$ in~$M$, the basic open
neighbourhood~$V_Y$ of~$f$ has non\dash{}empty intersection with~$G$, so
there is some~$g\in G$ such that $g\restriction_Y = f\restriction_Y$ and
consequently
$\phi_1(f)(b) = \phi_1(g)(b) = \phi_2(g)(b) = \phi_2(f)(b)$, where the
middle equality follows from
$\phi_1\Restriction_G = \phi_2\Restriction_G$. This is true for
any~$b\in B$, so $\phi_1(f)=\phi_2(f)$; however, also~$f\in M$  was
arbitrary and~$\phi_1$ and~$\phi_2$ were surjective, so $M_1= M_2$ and
$\phi_1=\phi_2$.
\par
For the converse
`$\eqref{item:grp-extension}\Rightarrow\eqref{item:aut-unif-homeo}$' we
assume condition~\eqref{item:grp-extension} and consider any monoid
isomorphism $\theta\colon M\to M'$ where $M'\in\mathcal{K}$ is a
closed transformation monoid on~$B$, $\lvert B\rvert=\lvert A\rvert$.
Let $G'\subs M'$ be the subset of invertibles of~$M'$. Certainly,
$G'\in\mathcal{G}$. Clearly, $G'$ is dense in its closure~$\cl{G'}$,
which belongs to~$\mathcal{K}$ by the assumption on~$\mathcal{K}$.
Indeed, $G'$ is also the set of invertibles of~$\cl{G'}$.
Since~$M'$ is closed, we have~$\cl{G'}\subs\cl{M'}=M'$.
Moreover, as~$\theta$ is an isomorphism, $\theta\fapply{G}\subs G'$, and
likewise~$\theta^{-1}\fapply{G'}\subs G$, so $\theta\fapply{G} = G'$.
Therefore, $\theta\restriction_G^{G'}\colon G\to G'$ is a well\dash{}defined
monoid isomorphism; moreover, as~$G$ and~$G'$ are group reducts,
$\theta\restriction_G^{G'}$ actually is a group isomorphism onto a group
in~$\mathcal{G}$. For~$G$ has automatic homeomorphicity
w.r.t.~$\mathcal{G}$, the isomorphism $\theta\restriction_G^{G'}$ is
a homeomorphism. By Proposition~\ref{prop:extending-dense-groups}, there
is an extension $\xi\colon M\to \cl{G'}$ of~$\theta\restriction_G^{G'}$,
which is a monoid isomorphism and a uniform homeomorphism. Since~$\xi$
and~$\theta$ coincide on~$G$, assumption~\eqref{item:grp-extension}
entails $\cl{\theta\fapply{G}}=\cl{G'} = M'$ and $\theta=\xi$. In
particular, $\theta$ is a uniform homeomorphism.
\end{proof}

\begin{remark}\label{rem:class-G-not-always-needed}
Later it will be useful to observe that the assumptions that~$G$ have
automatic homeomorphicity w.r.t.~$\mathcal{G}$ and that
$\lset{\cl{H}}{H\in \mathcal{G}}\subs\mathcal{K}$ were only needed to
prove the implication
`$\eqref{item:grp-extension}\Rightarrow\eqref{item:aut-unif-homeo}$'.
The `forward' implications
`$\eqref{item:aut-unif-homeo}\Rightarrow\eqref{item:aut-homeo}
                         \Rightarrow\eqref{item:weak-aut-homeo}
                         \Rightarrow\eqref{item:grp-extension}$'
hold even without these preconditions, so for them the
class~$\mathcal{G}$ does not play any role.
\end{remark}

With the help of the following lemma, we can
reobtain~\cite[Lemma~12]{Bodirsky}.
\begin{lemma}\label{lem:inj-mon-endo}
 Suppose that $M\subseteq \Ona{1}{A}$ and
 $M_1\subs M_2\subseteq \Ona{1}{B}$ are transformation monoids on
 sets~$A$ and~$B$, respectively, and let $G\subs M$ be any
 subset such that
 $E_G \defeq \lset{\psi\in\End(M)}{\psi \text{ injective} \land
 \psi\restriction_G = \id_M\restriction_G} = \set{\id_M}$, that is,
 the only injective monoid endomorphism of~$M$
 fixing~$G$ pointwise is the identity.
 If $\phi\colon M\to M_2$ is a monoid isomorphism and
 $\xi\colon M\to M_1$ is an injective monoid homomorphism such that
 $\xi\Restriction_G= \phi\Restriction_G$, then
 $\xi\fapply{M} = M_1=M_2$ and $\xi = \phi$.
 \par
 Moreover, for a monoid $M\subs\Ona{1}{A}$, a set $G\subs M$ and any
 class~$\mathcal{K}$ of monoids such that
 $\mathcal{K}\supseteq \lset{\psi\fapply{M}}{\psi\in E_G}$,
 condition~\eqref{item:grp-extension} of
 Proposition~\ref{prop:char-aut-homeo-dense-invertibles} is equivalent
 to $E_G = \set{\id_M}$.
\end{lemma}
\begin{proof}
Applying the identical monoid embedding $\iota\colon M_1\to M_2$, it
follows that $\psi\defeq \phi^{-1}\circ \iota\circ \xi$ is an injective
monoid endomorphism fixing every $g\in G$ since
$\psi(g) = \phi^{-1}(\xi(g)) = \phi^{-1}(\phi(g)) = g$.
Thus, $\psi=\id_M$, and hence $\phi = \phi\circ\psi=\iota\circ\xi$.
This means $M_2 = \phi\fapply{M} = \xi\fapply{M}\subs M_1\subs M_2$,
giving $\xi\fapply{M} = M_1 = M_2$ and $\xi=\phi$.
\par
The fact just shown clearly entails part~\eqref{item:grp-extension} of
Proposition~\ref{prop:char-aut-homeo-dense-invertibles} (for any class
of monoids); therefore, the latter is necessary given
$E_G = \set{\id_M}$. Conversely, knowing that
$\mathcal{K}\supseteq \lset{\psi\fapply{M}}{\psi\in E_G}$, for any
$\psi\in E_G$, we let $M_1 = \psi\fapply{M}$, $M_2 = M$,
$\psi_1 = \psi\restriction_M^{M_1}$ and $\psi_2 = \id_M$ in
Proposition~\ref{prop:char-aut-homeo-dense-invertibles}\eqref{item:grp-extension}.
We then conclude that $\psi\fapply{M} = M$, so $\psi$ is an isomorphism,
and $\psi = \psi_1 = \psi_2 = \id_M$.
\end{proof}

Lemma~12 of~\cite{Bodirsky} is the special case of the following result
where~$A$ is countably infinite.
\begin{corollary}\label{cor:automatic-homeo}
Supposing that $M\subs \Ona{1}{A}$ is a closed monoid, its group
$G \subs M$ of invertibles is dense in~$M$ and has automatic
homeomorphicity
and every injective monoid
endomorphism $\theta\colon M\to M$ with $\theta(g) = g$ for $g\in G$ is
the identity, then~$M$ has automatic homeomorphicity.
\end{corollary}
\begin{proof}
Choosing~$\mathcal{K}$ as the class of all closed transformation monoids
on carriers of the same cardinality as~$A$,
Lemma~\ref{lem:inj-mon-endo} demonstrates
condition~\eqref{item:grp-extension} of
Proposition~\ref{prop:char-aut-homeo-dense-invertibles}, so the claim
follows from statement~\eqref{item:aut-homeo} of the same result.
\end{proof}

\begin{remark}\label{rem:equivalence-of-inj-mon-endo}
Let $\mathcal{K}$, $\mathcal{G}$, $M$ and $G\subs M$ be as in
Proposition~\ref{prop:char-aut-homeo-dense-invertibles} and assume
$\mathcal{K}\supseteq \lset{\psi\fapply{M}}{\psi\in E_G}$ for
$E_G =\lset{\psi\in\End(M)}{\psi \text{ injective} \land
 \psi\restriction_G = \id_M\restriction_G}$.
It follows from Proposition~\ref{prop:char-aut-homeo-dense-invertibles}
and Lemma~\ref{lem:inj-mon-endo} that $E_G=\set{\id_M}$ is
actually equivalent to~$M$ having automatic homeomorphicity with respect
to~$\mathcal{K}$.
\par
Note that under the assumptions of Corollary~\ref{cor:automatic-homeo}
it does not automatically follow that the implication of the corollary
is an equivalence. This is because the endomorphisms $\psi\in E_G$ are
not necessarily closed maps, and therefore their images do not
necessarily belong to the class~$\mathcal{K}$ of closed transformation
monoids over sets equipotent with the carrier of~$M$. This is the same
type of complication that has made additional arguments necessary in
proving automatic homeomorphicity of $\End(\Q,\leq )$,
see~\cite[Lemma~4.1, p.~79]{BehReconstructingTopologyOnRationals} and
the discussion next to it.
\end{remark}

\section{Stronger reconstruction for monoids and clones}%
\label{sect:monoids}\label{sect:clones}
In this section we will show how to lift automatic action compatibility
from permutation groups first to endomorphism monoids and then to
clones. At our point of departure, we recall that
Rubin in~\cite[Theorem~2.2, p.~228]{Rubin1994} shows that any
\nbdd{\aleph_0}categorical structure without algebraicity
is \emph{group categorical} with respect to the class of all such
structures, provided it has weak
\nbdd{\forall\exists}interpretations. The exact meaning of this is
stated in Theorem~\ref{thm:Rubin-group-categoricity}, saying that
for every isomorphism $\phi\colon \Aut(\mathbb{A})\to\Aut(\mathbb{B})$
between the automorphism groups of countable
\nbdd{\aleph_0}categorical structures without algebraicity where
$\mathbb{A}$ has a weak \nbdd{\forall\exists}interpretation there is a
bijection $\theta\colon A\to B$ between the carrier sets~$A$ and~$B$
of~$\mathbb{A}$ and~$\mathbb{B}$, respectively, such that
$\phi(g) = \theta\circ g\circ \theta^{-1}$ for all
$g\in\Aut(\mathbb{A})$.
\par

On the other hand from Proposition~\ref{prop:extending-dense-groups},
we have that for transformation monoids $M\subs \Ona{1}{A}$ and
$M'\subs \Ona{1}{B}$ with dense groups of invertibles~$G$ and~$G'$ on
sets~$A$ and~$B$, any topological isomorphism $\xi\colon G \to G'$ can be
extended to a uniform homeomorphism and isomorphism $\exi\colon M\to M'$.
With the additional knowledge about~$\xi$ obtained from results such
as Theorem~\ref{thm:Rubin-group-categoricity} we can describe even more
precisely how $\exi\apply{f}$ for $f\in M$ will look like.

\begin{lemma}\label{lem:conj-grp-to-mon}
 Assume that $M\subseteq \Ona{1}{A}$ and $M'\subseteq \Ona{1}{B}$ are
 closed transformation monoids on carrier sets $A$ and $B$ with
 subsets $G\subseteq M$ and $G'\subs M'$ of invertible elements such
 that~$G$ is dense in~$M$.
 Moreover, let $\phi\colon M\to M'$ be a monoid isomorphism. If
 \begin{enumerate}[\upshape (i)]
  \item\label{item:test-equality-on-groups}
        for any isomorphism $\psi\colon M\to \cl{G'}$ onto the
        closure of~$G'$ in~$\Ona{1}{B}$ the condition
        $\psi\Restriction_G = \phi\Restriction_G$
        implies
        $\psi(f) = \phi(f)$ for all $f\in M$, and
  \item\label{item:restricted-grp-iso-is-induced}
        the isomorphism $\xi=\phi\restriction_G^{G'}\colon G\to G'$
        is induced by conjugation by some bijection $\theta\colon A\to B$,
 \end{enumerate}
then the latter fact extends to the monoid isomorphism~$\phi$; in fact,
$\phi$ is induced by conjugation by the same~$\theta\colon A\to B$.
\end{lemma}
Observe that with the help of Lemma~\ref{lem:inj-mon-endo} the first
condition of the preceding lemma can be replaced by the stronger
assumption that every injective monoid endomorphism~$\psi\colon M\to M$
fixing~$G$ pointwise is the identity. Moreover, it is clear from the
second condition that the sets~$A$ and~$B$ necessarily have to be
equipotent.
\begin{proof}
Invertibility is preserved under monoid isomorphisms, so
$\phi\fapply{G} = G'$, and thus the restriction $\xi\colon G\to G'$ is a
well\dash{}defined monoid isomorphism. By the second assumption it is a
uniform homeomorphism. The set of invertibles
of~$\cl{G'}$ is again~$G'$, and it is dense in the closed
monoid~$\cl{G'}$.
Thus, by Proposition~\ref{prop:extending-dense-groups},
the isomorphism~$\xi$ extends to a monoid isomorphism
$\exi\colon M\to\cl{G'}$ that also is a uniform homeomorphism.
Using the first assumption of the lemma, we infer that $\phi=\exi$ and
$\cl{G'}=M'$ since~$\phi$ and~$\exi$ are surjective.
Now, by our second assumption we can find some bijection
$\theta\colon A\to B$ inducing~$\xi$. It only remains to lift this
condition from~$\xi$ to~$\exi$ and thus to~$\phi$. This can be done by
only relying on the continuity of~$\exi$, but it is shorter to use the
explicit description of~$\exi$ from the proof of
Proposition~\ref{prop:extending-dense-groups}.
\par
Let $f\in M$ and $b\in B$. According to the construction given by
Proposition~\ref{prop:extending-dense-groups}, we take some finite
set~$C_b\subs A$ such that
$g_1\restriction_{C_b} = g_2\restriction_{C_b}$ implies
$\xi(g_1)(b) = \xi(g_2)(b)$ for any $g_1,g_2\in G$.
We put $\tilde{C}=C_b\cup\set{\theta^{-1}(b)}$ and consider any $g\in G$
such that $f\restriction_{\tilde{C}} = g\restriction_{\tilde{C}}$.
Then we observe
\[\phi(f)(b) = \exi(f)(b) = \xi(g)(b)
= \bigl(\theta\circ g\circ \theta^{-1}\bigr)(b)
= \theta\bigl(g\bigl(\theta^{-1}(b)\bigr)\bigr)
= \theta\bigl(f\bigl(\theta^{-1}(b)\bigr)\bigr),
\]
where the second equality holds by definition of~$\exi$, the third one
by the assumption on~$\xi$ and the last one since
$\theta^{-1}(b)\in\tilde{C}$.
This implies $\phi(f) = \theta\circ f\circ \theta^{-1}$.
\end{proof}

With the preceding lifting lemma we can transfer automatic action
compatibility from dense groups to monoids.

\begin{corollary}\label{cor:transfer-aut-act-comp}
Let $\mathcal{K}$ be a class of closed monoids and
$\mathcal{G}=\lset{G(M)}{M\in\mathcal{K}}$ be the corresponding class
of groups~$G(M)$ of invertibles of monoids~$M\in \mathcal{K}$. Moreover,
we assume that $\lset{\cl{G}}{G\in\mathcal{G}}\subs\mathcal{K}$, where
$\cl{G}$ denotes the closure in the full transformation monoid.
Supposing that $M\subs \Ona{1}{A}$ is a closed monoid with automatic
homeomorphicity w.r.t.~$\mathcal{K}$ and that its group
$G = G(M)$ of invertibles is dense in~$M$, then, provided~$G$ has
automatic action compatibility w.r.t.~$\mathcal{G}$, also~$M$ has
automatic action compatibility w.r.t.~$\mathcal{K}$.
\end{corollary}
It is possible to marginally weaken the assumption that~$M$ have
automatic homeomorphicity w.r.t.~$\mathcal{K}$ because, in the proof,
condition~\eqref{item:test-equality-on-groups} of
Lemma~\ref{lem:conj-grp-to-mon} will instantiate the universally
quantified implication in statement~\eqref{item:grp-extension} from
Proposition~\ref{prop:char-aut-homeo-dense-invertibles} only for
specific pairs $M_1\subs M_2$ from~$\mathcal{K}$. However the
modification would be rather technical and probably have only very few
applications.
\begin{proof}[Proof of Corollary~\ref{cor:transfer-aut-act-comp}]
Given any $M'\in\mathcal{K}$ on some set~$B$ that is equipotent with~$A$
and any monoid isomorphism~$\phi\colon M\to M'$ we shall use
Lemma~\ref{lem:conj-grp-to-mon} to prove that~$\phi$ is induced by
conjugation. By the choice of~$M$ and~$\mathcal{K}$ both monoids are
closed and~$G\subs M$ is dense. Also, since $G' = G(M')\in\mathcal{G}$
and~$G$ has automatic action compatibility w.r.t.~$\mathcal{G}$, the
isomorphism $\phi\restriction_G^{G'}$ is induced by conjugation; so
condition~\eqref{item:restricted-grp-iso-is-induced} of
Lemma~\ref{lem:conj-grp-to-mon} is satisfied.
Condition~\eqref{item:test-equality-on-groups} now follows from
Proposition~\ref{prop:char-aut-homeo-dense-invertibles}%
            \eqref{item:grp-extension}
since $\lset{\cl{G}}{G\in\mathcal{G}}\subs\mathcal{K}$, $M$ has
automatic homeomorphicity w.r.t.~$\mathcal{K}$ and~$G$ has automatic
homeomorphicity w.r.t.~$\mathcal{G}$ for it even has automatic action
compatibility w.r.t.~$\mathcal{G}$.
\end{proof}

Proving, for certain clones, automatic action compatibility w.r.t.\ all
\nbdd{\aleph_0}categorical structures without algebraicity will be based
on the next theorem.
The technique to prove it is inspired by~\cite{Rubin2016} and uses the
assumption of being \emph{weakly directed} (this notion has appeared
in~\cite[10.1, p.~60]{Rubin2016} as \emph{semi\dash{}transitive}, but the
latter term has been introduced in~\cite{semitransitive} to designate a
different semigroup property recurring in a number of
articles,
e.g.~\cite{semitransitive5,semitransitive2,semitransitive3,semitransitive4,semitransitive6,semitransitive7}).
Hence, we say that an action of a semigroup~$S$ on a
set~$A$ is weakly directed if for all $a,b\in A$ there are $f,g\in S$ and
$c\in A$ such that $(f,c) \mapsto a$ and $(g,c)\mapsto b$.
We call a transformation semigroup weakly directed if its action by
evaluation at points of the underlying set has this property. Certainly
every transitive action is weakly directed. Moreover, a straightforward
inductive argument shows that a weakly directed action of~$S$ on a
non\dash{}empty set~$A$ for every $n\in\N$ and $a_1,\dotsc,a_n\in A$
allows for finding $f_1,\dotsc,f_n\in S$ and $a_0\in A$ such that
$(f_i,a_0)\mapsto a_i$ holds for all $1\leq i\leq n$.

\begin{theorem}\label{thm1:reconstruction}
Let $F\subs \Ops$ and $F'\subs \OA{B}$ be clones on carrier sets $A$, $B$ such that
$\Fn[1]{F}$ is weakly directed and let $\xi\colon F\to F'$ be a
surjective clone\dash{}homomorphism such that the restriction
$\xi\restriction_{\Fn[1]{F}}^{\Fn[1]{F'}}\colon \Fn[1]{F} \to \Fn[1]{F'}$
is given by conjugation with some bijection
$\theta\colon A \to B$, i.e.,
\[\forall f\in \Fn[1]{F}\colon\quad \xi\apply{f}=\theta \circ f \circ \theta^{-1},\]
then~$\xi$ is induced by conjugation by the same~$\theta$, so
\[\forall n\in \N\, \forall h\in \Fn{F}\colon \quad \xi\apply{h}=\theta
\circ h \circ \apply{\theta^{-1}\times \cdots \times \theta^{-1}}.\]
In particular, $\xi$ is a clone isomorphism and a uniform homeomorphism.
\end{theorem}
Note that the following proof also works for nullary operations, which
only exist on non\dash{}empty carrier sets.
Moreover, there is absolutely no restriction on the cardinality of the
carriers~$A$ and~$B$ here.
\begin{proof}
If $A=\emptyset$, the assumed bijection~$\theta\colon A\to B$ ensures
that~$B=\emptyset$ and thus $F=F'=\OA{\emptyset}$, which contains only
projections. So the claim is trivially true.
\par
Now let~$A\neq \emptyset$,
let $n\in \N$, $h\in \Fn{F}$, and consider any $y_1,y_2,\ldots,y_n\in B$.
We put $a_i=\theta^{-1}\apply{y_i}$ for all $1\leq i\leq n$.
Since~$\Fn[1]{F}$ is weakly directed, there is some $a\in A$ and $g_1,\ldots,g_n\in \Fn[1]{F}$ such that $g_i(a)=a_i$ for all $1\leq i\leq n$.\par
Consider $f\defeq h\circ\apply{g_1,\ldots g_n}\in\Fn[1]{F}$. By the assumption on
$\xi\restriction_{\Fn[1]{F}}^{\Fn[1]{F'}}$, we have

\[\xi\apply{f}=\theta \circ f \circ \theta^{-1}= \theta \circ h \circ \apply{g_1,\ldots, g_n} \circ \theta^{-1},\]
on the other hand, since $\xi$ is compatible with $\circ$, we have
\[\xi\apply{f}=\xi\apply{h\circ \apply{g_1,\ldots,g_n}}= \xi(h)\circ\apply{\xi(g_1),\ldots,\xi(g_n)},\]
and again because of the assumption on
$\xi\restriction_{\Fn[1]{F}}^{\Fn[1]{F'}}$, we have
\[\xi\apply{f}=\xi\apply{h}\circ
\apply{\xi\apply{g_1},\ldots,\xi\apply{g_n}}= \xi(h)\circ\apply{\theta\circ
g_1\circ\theta^{-1},\ldots,\theta\circ g_n \circ \theta^{-1} }.\]
Finally, we evaluate $\xi(f)$ at $y\defeq \theta\apply{a}$:
\begin{align*}
\xi\apply{f}(y)&=\theta\apply{h\apply{g_1\apply{\theta^{-1}(y)},\ldots, g_n\apply{\theta^{-1}(y)}}}=\theta\apply{h\apply{g_1\apply{a},\ldots,g_n\apply{a}}}\\ &= \theta\apply{h\apply{a_1,\ldots,a_n}}=\theta\apply{h\apply{\theta^{-1}(y_1),\ldots,\theta^{-1}(y_n)}}
\end{align*}
and
\begin{align*}
\xi\apply{f}(y)&=\xi(h)
\apply{\theta\apply{g_1\apply{\theta^{-1}(y)}},\ldots,\theta\apply{g_n\apply{\theta^{-1}(y)}}}\\
&=
\xi(h)
\apply{\theta\apply{g_1(a)},\ldots,\theta\apply{g_n\apply{a}}}\\
&=\xi(h)\apply{\theta\apply{a_1},\ldots,\theta\apply{a_n}}=\xi\apply{h}\apply{y_1,\dots,y_n}.
\end{align*}
Because~$\xi$ is surjective, it is an isomorphism and a uniform
homeomorphism.
\end{proof}

\begin{remark}
If $F\subs \Ops$ is a closed clone where $\Fn[1]{F}$ is weakly directed
and $F'\subs\OA{B}$ is a clone that fails to be closed, then there is no
surjective clone homomorphism $\xi\colon F\to F'$ whose restriction
$\xi\restriction_{\Fn[1]{F}}^{\Fn[1]{F'}}$ to the monoid parts is induced
by some bijection $\theta\colon A\to B$. Otherwise,
Theorem~\ref{thm1:reconstruction} would imply that $\xi$ is a
homeomorphism and hence, $F' = \xi\fapply{F}$ would have to be closed.
\end{remark}

\begin{theorem}\label{thm2:reconstruction}
 Let~$\mathcal{K}$ and~$\mathcal{C}$ be classes of transformation
 monoids and clones, respectively, such that
 $\set{\Fn[1]{F'}\mid F'\in\mathcal{C}}\subseteq \mathcal{K}$.
 Moreover, let $F\subs \Ops$ be a clone with a
 weakly directed monoid~$\Fn[1]{F}$ having automatic action
 compatibility w.r.t.~$\mathcal{K}$.
 Then~$F$ has automatic action compatibility \inbr{and thus automatic
 homeomorphicity} w.r.t.\ the class~$\mathcal{C}$.
\end{theorem}
Note that this theorem only makes, albeit strong, assumptions on the
monoid part~$\Fn[1]{F}$ of the clone under consideration.
Furthermore, no closedness requirements are made. They may, however, be
necessary to provide the preconditions of the theorem for concrete
instances.
\begin{proof}
 Let $\xi\colon F\to F'$ be a clone isomorphism where $F'\in\mathcal{C}$
 is a clone on a set~$B$ of the same size as~$A$. By assumption,
 $\xi\restriction_{\Fn[1]{F}}^{\Fn[1]{F'}}\colon \Fn[1]{F}\to
 \Fn[1]{F'}$ is a monoid isomorphism onto $\Fn[1]{F'}\in\mathcal{K}$.
 Thus, it is given by conjugation by some bijection
 $\theta\colon A\to B$,
 which of course may depend on~$\xi$ and $\Fn[1]{F'}$.
 Applying Theorem~\ref{thm1:reconstruction}, $\xi$ is induced by
 conjugation by~$\theta$, as well, and~$\xi$ is a uniform homeomorphism.
\end{proof}

\section{Automatic action compatibility for concrete
structures}\label{sect:examples}
We begin with a convenience result summarizing a set of
assumptions that allows to combine all the previous results in a smooth
manner. Other ways to put Theorem~\ref{thm1:reconstruction} to work (with
different assumptions) are certainly conceivable.
\begin{corollary}\label{cor:combining-results}
Let~$\mathcal{K}$ be the class of all endomorphism monoids of countable
\nbdd{\aleph_0}categorical structures without algebraicity, and
let~$\mathbb{A}$ be such a structure.
Let~$M$ be a closed transformation monoid
on the carrier set of~$\mathbb{A}$, for instance, $M=\End(\mathbb{A})$
or $M=\Emb(\mathbb{A})$. If
\begin{enumerate}[\upshape (1)]
\item\label{item:forallexists}
  $\mathbb{A}$ has a weak \nbdd{\forall\exists}interpretation,
\item\label{item:density}
  $\Aut(\mathbb{A})$ is dense in~$M$ and coincides with the group of
  invertible elements $\set{g\in M \mid \exists f\in M\colon f\circ g = g\circ f
  = \id_A}$,
\item\label{item:weakly-directed}
  $M$ is weakly directed, e.g.\ transitive, and
\item\label{item:inj-mon-end}
  every injective monoid endomorphism of~$M$ that
  fixes $\Aut(\mathbb{A})$ pointwise is the identity,
  or\par
  $M$ has automatic homeomorphicity w.r.t.\ a
  class~$\mathcal{L}\supseteq\mathcal{K}$ of closed transformation
  monoids such that $\cl{G'}\in\mathcal{L}$ for the set~$G'$ of
  invertibles of any monoid~$M'\in\mathcal{K}$,
\end{enumerate}
then any closed clone~$F$ on the carrier set
of~$\mathbb{A}$ satisfying $\Fn[1]{F} = M$ has automatic action
compatibility \inbr{and thus automatic homeomorphicity} with respect to the
class~$\mathcal{C}$ of polymorphism clones of countable
\nbdd{\aleph_0}categorical structures without algebraicity.
Moreover, $M$ has automatic action compatibility with respect
to~$\mathcal{K}$.
\end{corollary}
\begin{proof}
Let us abbreviate $G =\Aut(\mathbb{A})$.
By assumption, the unary part of our clone, $\Fn[1]{F}=M$,
is weakly directed. Let~$\mathcal{K}$ and~$\mathcal{C}$ be as
described.
Clearly, $\set{\Fn[1]{F'}\mid F'\in\mathcal{C}}\subseteq \mathcal{K}$.
Moreover, let $\xi\colon M\to M'$ be any monoid
isomorphism onto any transformation monoid $M'\in\mathcal{K}$ on some
set equipotent to~$A$. As soon as we verify that~$\xi$ is induced by
conjugation, Theorem~\ref{thm2:reconstruction} will yield the desired
conclusion.
\par
To achieve this, we shall employ Lemma~\ref{lem:conj-grp-to-mon}.
Certainly, $M'$ is closed since
$M'\in\mathcal{K}$. Likewise, $M$ is a closed monoid with dense group of
invertibles~$G$. Furthermore, let $\psi\colon M\to \cl{G'}$ be a monoid
isomorphism onto the closure of the invertibles~$G'$ inside~$M'$ that
coincides with~$\xi$ on members of~$G$. If every injective monoid
endomorphism of~$M$ that fixes~$G$ pointwise is the identity, then
Lemma~\ref{lem:inj-mon-endo} states that~$\xi$ equals~$\psi$. Otherwise,
if~$M$ has automatic homeomorphicity w.r.t.~$\mathcal{L}$, then
Proposition~\ref{prop:char-aut-homeo-dense-invertibles}\eqref{item:aut-homeo}
is satisfied (cf.\ Remark~\ref{rem:class-G-not-always-needed}).
So Proposition~\ref{prop:char-aut-homeo-dense-invertibles}\eqref{item:grp-extension} yields that $\xi=\psi$ since the
closed monoids $\cl{G'}\subs M'$ belong to~$\mathcal{L}$.

Finally,
as~$\mathbb{A}$ is \nbdd{\aleph_0}categorical without algebraicity and
has a weak \nbdd{\forall\exists}interpretation,
Theorem~\ref{thm:Rubin-group-categoricity} yields that the restriction
$\xi\restriction_G^{G'}\colon G\to G'$ of~$\xi$ to~$G$, is
induced by conjugation. This follows for,
by the construction of~$\mathcal{K}$, $G'$ is an automorphism group of a
countable \nbdd{\aleph_0}categorical structure without
algebraicity and thus covered by
Theorem~\ref{thm:Rubin-group-categoricity}.
\end{proof}

It is interesting to observe that the main structural restrictions
(countable categoricity, absence of algebraicity) for the preceding
result come from the group case, that is,
Theorem~\ref{thm:Rubin-group-categoricity}. This means a stronger result
regarding the automorphism groups would allow for a wider ranging
reconstruction result regarding the clones.
Next we state a few less technical assumptions, allowing us to use
Corollary~\ref{cor:combining-results}.
\begin{corollary}\label{cor:easier-assumptions}
Let $\mathbb{A}$ be a countable \nbdd{\aleph_0}categorical homogeneous
relational structure without algebraicity and $M=\Emb(\mathbb{A})$.
If $\Aut(\mathbb{A})$ is transitive and supports a weak\/
\nbdd{\forall\exists}interpretation, then
Corollary~\ref{cor:combining-results} applies to~$\mathbb{A}$ and\/
$\Emb(\mathbb{A})$, which in this case coincides with the monoid of elementary
self\dash{}embeddings of\/~$\mathbb{A}$.
\end{corollary}
We note that our argument for
condition~\eqref{item:inj-mon-end} is somewhat similar to the
strategy employed in the proof of~\cite[Theorem~4.7,
p.~21]{PechPolymorphismClonesHomStr}.
\begin{proof}
Since~$\mathbb{A}$ is countable and \nbdd{\aleph_0}categorical, it is
saturated (see~\cite[Example~5, p.~485]{HodgesModelTheory}).
For this reason, the automorphism group is dense in the monoid of
elementary self\dash{}embeddings
(cf.~\cite[p.~598]{PechReconstructingTopologyOfEEmbCntblSaturatedStr})
and, by homogeneity, its closure coincides with~$M=\Emb(\mathbb{A})$,
whose invertibles are exactly~$\Aut(\mathbb{A})$
(see Lemma~\ref{lem:hom-structures}\eqref{item:emb-dense}\eqref{item:emb-cmplmt}).
Transitivity of~$\Aut(\mathbb{A})$ is inherited by~$M$.
For~$\mathbb{A}$ avoids any algebraicity,
Lemma~\ref{lem:hom-structures}\eqref{item:trivial-centre} gives that
$\Aut(\mathbb{A})$ has a trivial centre. As~$\mathbb{A}$ is countable
and saturated, and~$M$ equals the monoid of elementary
self\dash{}embeddings, Proposition~2.5
of~\cite[p.~600]{PechReconstructingTopologyOfEEmbCntblSaturatedStr}
provides the monoid condition~\eqref{item:inj-mon-end} of
Corollary~\ref{cor:combining-results} for~$M$.
\end{proof}

Subsequently, we consider a list of example structures.
With two exceptions they satisfy all assumptions of
Corollary~\ref{cor:combining-results}; for the two particular
cases only one condition will remain open.
\par
As our first examples, we consider the non\dash{}trivial
first\dash{}order reducts of the rationals~$\apply{\Q,<}$ studied
in~\cite{CameronTransitivityOfPermutationGroupsOnUnorderedSets}.
Each of them is given by a single relation:
$(\Q, \betw)$, $(\Q, \crc)$, and $(\Q, \sep)$, where for any elements
$x,y,z,t\in\Q$ we have
\begin{align*}
\betw(x, y, z)& \iff x < y < z \lor z < y < x,\\
\crc(x, y, z)& \iff x < y < z \lor y < z < x \lor z < x < y,\\
\sep(x, y, z, t)& \iff \apply{\crc(x, y, z) \land \crc(x, t, y)}\lor
                       \apply{\crc(x, z, y) \land \crc(x, y, t)}.
\end{align*}
These reducts featured prominently in the complexity classification of
temporal constraint
languages~\cite{BodirskyKaraComplexityOfTemporalCSP}.
Moreover, a selection of further notorious candidates from the zoo of
countable universal homogeneous structures will play a role.

\begin{lemma}\label{lem:str-partially-satisfying-assumptions}
Let $\mathbb{A}$ be one of the following structures:
\begin{enumerate}[\upshape (a)]
\item a reduct~$(\Q,\rho)$ of the strictly ordered rationals~$(\Q,<)$,
      where $\rho$ is one of the relations in
      $\set{\mathord{<},\betw,\crc,\sep}$;
\item the countable universal homogeneous poset~$\mathbb{P}$ \inbr{under the strict order};
\item the Rado \inbr{also random} graph~$\mathbb{G}$;
\item the random directed graph~$\mathbb{D}$;
\item the universal homogeneous version of the random bipartite
      graph~$\mathbb{B}$
      \inbr{with an additional relation for the bipartition,
      cf.~\cite[p.~1603]{MacphersonHomogeneousStructures}};
\item the countable universal homogeneous tournament~$\mathbb{T}$;
\item the countable dense local order~$\mathbb{S}_2$,
      see~\cite{CameronOrbitsOfPermutationGroupsOnUnorderedSetsII}
      or e.g.~\cite[p.~1604~(ii)]{MacphersonHomogeneousStructures};
\item Cherlin's myopic local order~$\mathbb{S}_3$,
      see~\cite{CherlinHomogeneousDirectedGraphsTheImprimitiveCase}
      or~\cite[Example 2.3.1~2., p.~1605]{MacphersonHomogeneousStructures};
\item the countable universal homogeneous \nbdd{k}uniform
      hypergraph~$\mathbb{H}_k$ for~$k\geq 2$;
\item the countable universal homogeneous \nbdd{\mathbb{K}_n}free
      graph~$\mathbb{G}_{-\mathbb{K}_n}$ for $n\geq 3$;
\item any of the countably universal homogeneous Henson
      digraphs~$\mathbb{D}_X$, forbidding a certain set~$X$ of finite
      tournaments, see
      e.g.~\cite[p.~1604~(iii)]{MacphersonHomogeneousStructures}.
\end{enumerate}
Then~$\mathbb{A}$ satisfies all
assumptions of~Corollary~\ref{cor:combining-results} other
than~\eqref{item:forallexists} with respect to
the closed monoid~$M=\Emb(\mathbb{A})$, which coincides with the monoid
of elementary self\dash{}embeddings of~$\mathbb{A}$.
\par
If $\mathbb{A}\notin\set{(\Q,\betw), (\Q,\crc), (\Q,\sep)}$, then
Corollary~\ref{cor:combining-results} is indeed applicable and yields
that~$M$ and
every closed clone~$F$ on the carrier of~$\mathbb{A}$ having~$M$ as its
unary part, both have automatic action compatibility with respect to
countable \nbdd{\aleph_0}categorical structures without algebraicity.
In particular, this holds for the clone $F=\Pol(\cpl{A})$; in the
case of the rationals also for $F=\Pol(\mathbb{A})=\Pol(\Q,<)$.
\end{lemma}
Note that in the case of the reducts $(\Q,\rho)$ one can also use the
endomorphism monoid, because $\End(\Q,\rho) = \Emb(\Q,\rho)$. The proof of
the latter fact is completely elementary, but (at least in the case
of~$\sep$) a long and tedious case distinction. So it is perfectly suited
to be left to an automated theorem prover.
%
\begin{proof}
Most of these structures~$\mathbb{A}$ are, by definition, limits of
Fra\"{\i}ss\'e classes, so they are countably infinite universal
homogeneous structures. The reducts $(\Q,\rho)$ and~$\mathbb{S}_2$,
$\mathbb{S}_3$
were not defined in this way, but nonetheless are homogeneous, see
e.g.~\cite[Example~2.3.1, p.~1605]{MacphersonHomogeneousStructures}.
Since~$\mathbb{A}$ has a finite relational signature, it is
\nbdd{\aleph_0}categorical
(Lemma~\ref{lem:hom-structures}\eqref{item:cntbl-cat}).
By part~\eqref{item:loopless} of the same lemma, $\Aut(\mathbb{A})$
is transitive because all given structures are `loopless' in
the sense that the intersection of each of their fundamental relations
with the appropriate~$\Delta^{(m)}_A$ is empty.
Moreover, all the listed structures have no algebraicity. To give some literature
references, for $\mathbb{H}_k$ this information can be obtained from the
proof of Corollary~22 of~\cite[p.~3726]{Bodirsky}. Moreover,
every structure listed by Rubin in~\cite{Rubin1994} as examples
for his Theorem~2.2 has this property, see p.~234 et seq.\ for
$\mathbb{D}_X$, $\mathbb{D}=\mathbb{D}_{\emptyset}$,
$\mathbb{G}_{-\mathbb{K}_n}$, $\mathbb{G}$, $\mathbb{B}$, $\mathbb{T}$
and p.~243 for~$\mathbb{P}$, $\mathbb{S}_2$, $\mathbb{S}_3$ and~$(\Q,<)$.
For $(\Q,<)$ has no algebraicity, the same is true for any of its
reducts as $\Aut(\Q,<) \subs \Aut(\Q,\rho)$ wherefore the reducts have
even bigger orbits than those given by stabilizers of $\Aut(\Q,<)$.
\par
Except for \nbdd{\forall\exists}interpretations we are now ready to
invoke Corollary~\ref{cor:easier-assumptions} to have
Corollary~\ref{cor:combining-results} deliver the conclusion. For this we assume
that $\mathbb{A}\neq (\Q,\rho)$ for any $\rho\in\set{\betw,\crc,\sep}$.
Concerning the remaining condition, Rubin proves
in~\cite[Theorem~3.2, p.~235]{Rubin1994} that every `simple' structure
has a weak \nbdd{\forall\exists}interpretation. The list of `simple'
structures given in~\cite[Examples(1)--(3), p.~234 et seq.]{Rubin1994}
covers~$\mathbb{G}$, $\mathbb{D}$, $\mathbb{B}$, $\mathbb{T}$,
$\mathbb{D}_X$ and $\mathbb{G}_{-\mathbb{K}_n}$. The sporadic
\nbdd{\forall\exists}interpretations given in~\cite[p.~243]{Rubin1994}
cover~$(\Q,<)$, $\mathbb{P}$, and $\mathbb{S}_2$ and $\mathbb{S}_3$.
Finally, $\mathbb{H}_k$ is treated
in~\cite{BarbinaMacphersonReconstructionHomRelStr2007}, see section~1 and
the discussion following~Theorem~4.1 ibid.
\par
Note that for $\mathbb{A} = (\Q,<)$, we have
$\Emb(\Q,\mathord{<})=\End(\Q,<) = \Pol^{(1)}(\Q,<)$, so that one can
avoid forming the structure $\cpl{A}$.
\end{proof}

\begin{remark}
For each of~$\mathbb{G}$, $\mathbb{D}$, $\mathbb{T}$ and~$\mathbb{H}_k$
\inbr{$k\geq 2$}, the monoid condition~\eqref{item:inj-mon-end} of
Corollary~\ref{cor:combining-results} also is a
consequence of~\cite[Lemma~20, p.~3726]{Bodirsky}, which is applicable
since these structures have the joint extension property, cf.\ the proof
of Corollary~22 \inbr{p.~3726} and the discussion after Definition~18
\inbr{p.~3724} in~\cite{Bodirsky}.
\par
Regarding the reducts~$(\Q,\rho)$,
for $\rho = \mathord{<}$, assumption~\eqref{item:inj-mon-end} was more
explicitly verified in~\cite[Corollary~2.5]{BehReconstructingTopologyOnRationals}.
For $\rho = \betw$, this condition is shown in the proof of
Theorem~2.3 of~\cite{TrussVargasReconstructingTopologyMonoidsPolClonesOfRationals}, for $\rho=\crc$,
it is stated in Corollary~3.5 of~\cite{TrussVargasReconstructingTopologyMonoidsPolClonesOfRationals}, and finally, for
$\rho=\sep$, this fact is discussed at the beginning of section~4
of~\cite{TrussVargasReconstructingTopologyMonoidsPolClonesOfRationals}, before~Theorem~4.1.
\end{remark}

It can be seen that for the reducts of $(\Q,<)$ the first condition
of Corollary~\ref{cor:combining-results} has been left open so far.
Christian Pech kindly pointed out to us that Silvia Barbina gave a weak
\nbdd{\forall\exists}interpretation for~$(\Q,\betw)$ in her PhD-thesis,
see~\cite[Example~1.5.3, p.~38]{BarbinaThesis}.
\begin{corollary}\label{cor:Q-betw}
Corollary~\ref{cor:combining-results} is applicable to
$\mathbb{A}=(\Q,\betw)$ and $M=\End(\mathbb{A})$. Hence, $M$ and any
closed clone~$F\subs\OA{\Q}$ with $\Fn[1]{F}=M$, e.g.,
$F=\Pol(\Q,\betw)$, has automatic action compatibility with respect to
countable \nbdd{\aleph_0}categorical structures without algebraicity.
\end{corollary}
\begin{proof}
Combine Lemma~\ref{lem:str-partially-satisfying-assumptions}
with~\cite[Example~1.5.3]{BarbinaThesis}, and recall that in this case
$\Emb(\mathbb{A}) = M = \Pol^{(1)}(\mathbb{A})$.
\end{proof}

Barbina's construction uses
that $\Aut(\Q,<)$ is a closed normal oligomorphic transitive subgroup
of~$\Aut(\Q,\betw)$ and that this subgroup is existentially definable
in~$\Aut(\Q,\betw)$ to transfer Rubin's
\nbdd{\forall\exists}interpretation for~$(\Q,<)$,
see~\cite[p.~243]{Rubin1994}, to~$(\Q,\betw)$. The detailed
requirements on the interpretation for when such a transfer is possible
are stated in~\cite[Proposition~1.5.1, p.~36]{BarbinaThesis}.
Subsequently we are going to show that such a transfer is also possible
between $(\Q,\crc)$ and~$(\Q,\sep)$. However, we currently are not
aware of \nbdd{\forall\exists}interpretations for~$(\Q,\crc)$, even
though~\cite[Proposition~1.2.9, p.~23]{BarbinaThesis} might possibly
provide a route to them. Let us record this problem explicitly:
\par
\begin{problem}\label{prob:reducts-forallexists}
Which of the reducts~$(\Q, \crc)$ and~$(\Q,\sep)$
have weak \nbdd{\forall\exists}interpretations? Does $(\Q,\crc)$ have a
weak \nbdd{\forall\exists}interpretation satisfying the additional
assumptions 1.--4.\ mentioned
in~\cite[Proposition~1.5.1]{BarbinaThesis}?
\end{problem}

For Barbina's transfer result the existential definability of the
smaller group in the bigger one can be particularly tricky. For this we
present the following lemma,
generalizing~\cite[Lemma~1.5.2, p.~37]{BarbinaThesis}.

\begin{lemma}\label{lem:local-generics}
Let~$G$ be a Polish group and $H\nsg G$ a closed normal subgroup that
has the following sort of locally generic
\inbr{cf.~\cite[p.~122]{TrussGenericAutomorphisms}}
elements $h\in H$: there is a non\dash{}empty open subset $X\subs H$
such that every $k\in H$ is a product of two members of~$X$ and there
is $h\in X$ \inbr{said to be locally generic}, the \nbdd{H}conjugacy class
$C\defeq C_H = \lset{ghg^{-1}}{g\in H}$
of which is comeagre in~$X$, that is, $X\cap C$ is a comeagre subset
of~$X$. Under these conditions, $H$ is existentially definable
\inbr{with parameter~$h$} in~$G$.
\end{lemma}
Note that $X\cap C\subs X$ being comeagre means that $X\setminus
C=X\setminus(X\cap C)$ is a meagre subset of~$X$. Here the notion of
being meagre does not differ whether it is understood with respect to
the topology of~$H$ or the topology of the subspace~$X$, since~$X$ is
open in~$H$. Moreover, $H$ is again Polish since it is closed in~$G$.
In the proof we shall use the following observation, which is readily
verified.
\begin{fact}\label{fact:product}
If~$G$ is a group, $g\in G$ and $S\subs G$ any subset, then
$g\in S\cdot S$ \inbr{i.e., $g$ is a product of two possibly equal
elements from~$S$} if and only if $S\cap gS^{-1}\neq \emptyset$.
\end{fact}
\begin{proof}[Proof of Lemma~\ref{lem:local-generics}]
Let $h\in X$, $X\subs H$ and~$C$ be as described above. We begin by
showing that $H\subs C\cdot C$ and thus consider an arbitrary element
$k\in H$. We work in the Polish group~$H$; since it is a topological
group, taking inverses or left or right multiplications by some element
of~$H$ are homeomorphisms of~$H$. We know that $X\setminus
C=X\setminus(X\cap C)$ is meagre in~$H$, so
$X^{-1}\setminus C^{-1}$, which can be written as $X^{-1}\setminus
\apply{C^{-1}\cap X^{-1}}= X^{-1}\setminus(C\cap X)^{-1}$ is meagre
in~$H$, too, and likewise
$kX^{-1}\setminus kC^{-1}=kX^{-1}\setminus\apply{kC^{-1}\cap kX^{-1}}
=kX^{-1}\setminus k(C\cap X)^{-1}$.
Subsets of meagre sets are again meagre, so
$\apply{X\cap kX^{-1}}\setminus C\subs X\setminus C$ and
$\apply{X\cap kX^{-1}}\setminus kC^{-1} \subs kX^{-1}\setminus kC^{-1}$
are both meagre, and so is their union
$\apply{X\cap kX^{-1}}\setminus \apply{C\cap kC^{-1}}$.
Exploiting the homeomorphisms, $X^{-1}$ and $kX^{-1}$ are open, and so
$X\cap kX^{-1}\subs H$ is open. By the assumption on~$X$ and
Fact~\ref{fact:product}, $X\cap kX^{-1}\neq \emptyset$; since~$H$ is
Polish, this open and non\dash{}empty set must be non\dash{}meagre.
If $C\cap kC^{-1}$ were empty, then
$X\cap kX^{-1}=\apply{X\cap kX^{-1}}\setminus\apply{C\cap kC^{-1}}$
would be meagre, thus we conclude that $C\cap kC^{-1}\neq \emptyset$.
By Fact~\ref{fact:product}, $k\in C\cdot C$.
\par
From $H\subs C\cdot C$ we continue as
in the proof of~\cite[Lemma~1.5.2]{BarbinaThesis}.
We define $C_G\defeq\lset{ghg^{-1}}{g\in G}\supseteq C$.
Since $H\nsg G$, we have $C_G\subs H$, so
$H\subs C\cdot C\subs C_G \cdot C_G \subs H\cdot H\subs H$, leading to
\begin{align*}
H = C_G\cdot C_G =
\lset{x\in G}{\exists g_1,g_2\in G\colon x=g_1hg_1^{-1}\cdot
g_2hg_2^{-1}},
\end{align*}
which is the evaluation of an existential \nbdd{G}formula with
parameter~$h$.
\end{proof}

While Barbina employed generic automorphisms to obtain a weak
\nbdd{\forall\exists}interpretation for $(\Q,\betw)$, we can now use
locally generic automorphisms.
\begin{lemma}\label{lem:Q-circ}
$H\defeq \Aut(\Q,\crc)$ is a closed oligomorphic normal subgroup of
index~$2$ of~$G\defeq \Aut(\Q,\sep)$, which acts transitively on~$\Q$
and is existentially definable with a single parameter in~$G$.
\end{lemma}
\begin{proof}
By the definition of~$\sep$ from~$\crc$, $H\subs G$ is a subgroup; it
is closed and oligomorphic since it is the automorphism group of a
countable \nbdd{\aleph_0}categorical structure. $H$ acts transitively
on~$\Q$ by Lemma~\ref{lem:hom-structures}\eqref{item:loopless}, and it
is normal since it has index~$2$ in~$G$.
The latter holds since~$G$ contains bijections preserving the circular
order, i.e., members of~$H$, and bijections~$f$ reversing the circular
order (in the sense that any $(x,y,z)\in\crc$ is mapped to
$(f(x),f(z),f(y))\in\crc$),
very much analogously like $\Aut(\Q,\betw)$ consists of order
preserving and order reversing bijections with respect to $(\Q,<)$.
Using orbit-stabilizer techniques and the observation that the
\nbdd{G}stabilizer of a single point~$a$ of~$\Q$ is $\Aut(\Q,\betw)$
with respect to a suitably defined betweenness relation on
$\Q\setminus\set{a}$ (some more details on this can be found
in~\cite[section~11.3.4, p.~110 et seq.]{BhattacharjeeMacphersonMoellerNeumannNotesOnInfinitePermutationGroups}), one can show
that~$G$ actually consists of none other than two disjoint cosets
of~$H$, namely~$H$ and $fH$ where~$f$ is some bijection reversing the
circular order, e.g., given by $f(x­)=-x$ for $x\in\Q$.
\par
To get that~$H$ is existentially definable in~$G$, we make use of
Lemma~\ref{lem:local-generics}. It is well known that the full symmetric
group on a countable set, such as~$\Q$, is a Polish group (see, e.g.,
\cite[section~2.6, p.~97]{MellerayPolishGroups}). Hence, any
closed subgroup of it, that is, any automorphism group of a countable
structure, e.g., $H$, is Polish, too. Now, according to Example~5.6
of~\cite[p.~134]{TrussGenericAutomorphisms}, there is some
\nbdd{\crc}preserving permutation
$h\in X = \lset{k\in H}{\exists q\in \Q\colon k(q)=q}$
which is a locally generic automorphism on the open subset~$X\subs H$.
Using again orbit-stabilizer methods to get a more explicit description
of the functions in~$H$, it can also be verified that every $k\in H$ is
a product of two elements from~$X$, i.e., of two automorphisms in~$H$
each having some fixed point.
\end{proof}

\begin{corollary}\label{cor:Q-circ-Q-sep}
If $(\Q,\crc)$ has a weak \nbdd{\forall\exists}interpretation
satisfying the conditions in Proposition~1.5.1.\,1.--4.
from~\cite[p.~36]{BarbinaThesis}, then $(\Q,\crc)$ and~$(\Q,\sep)$ each
will have one, and will satisfy the assumptions of
Corollary~\ref{cor:combining-results}, which will entail automatic
action compatibility with respect to countable
\nbdd{\aleph_0}categorical structures of the embeddings monoid and any
closed clone having this monoid as its unary part.
\end{corollary}
\begin{proof}
The assumed weak \nbdd{\forall\exists}interpretation for $(\Q,\crc)$
and Lemma~\ref{lem:Q-circ} provide the preconditions
for~\cite[Proposition~1.5.1]{BarbinaThesis}. This result now yields
that $(\Q,\sep)$ has a weak \nbdd{\forall\exists}interpretation, too,
and this, together with the facts shown in
Lemma~\ref{lem:str-partially-satisfying-assumptions}, will make
Corollary~\ref{cor:combining-results} applicable.
\end{proof}

We mentioned earlier that there may be different ways to combine our
theorems to obtain reconstruction results. In this way
Theorem~\ref{thm:reconstruction-ssip}, recently proved by Paolini and
Shelah (see~\cite{PaoliniShelahReconstructionWithStrongSIP}), can be
used to provide reconstruction of the action for the reducts
$(\mathbb{Q},\crc)$ and $(\mathbb{Q},\sep)$, with respect to a
slightly smaller class than the one we would get from a solution to
Problem~\ref{prob:reducts-forallexists}.

\begin{corollary}\label{cor:combining-results-ssip}
Let~$\mathcal{K}'$ be the class of all endomorphism monoids of countable
\nbdd{\aleph_0}categorical structures with the strong small index
property and no algebraicity, and let~$\mathbb{A}$ be such a structure.
Let~$M$ be a closed transformation monoid
on the carrier set of~$\mathbb{A}$, e.g.\ $M=\End(\mathbb{A})$
or $M=\Emb(\mathbb{A})$. If
\begin{enumerate}[\upshape (1)]
\item\label{item:density-ssip}
  $\Aut(\mathbb{A})$ is dense in~$M$ and coincides with the group of
  invertible elements $\set{g\in M \mid \exists f\in M\colon f\circ g = g\circ f
  = \id_A}$,
\item\label{item:weakly-directed-ssip}
  $M$ is weakly directed, e.g.\ transitive, and
\item\label{item:inj-mon-end-ssip}
  every injective monoid endomorphism of~$M$ that
  fixes $\Aut(\mathbb{A})$ pointwise is the identity,
  or\par
  $M$ has automatic homeomorphicity w.r.t.\ a
  class~$\mathcal{L}\supseteq\mathcal{K}'$ of closed
  transformation monoids such that
  $\cl{G'}\in\mathcal{L}$ for the set~$G'$ of invertibles of
  any monoid $M'\in\mathcal{K}'$,
\end{enumerate}
then any closed clone~$F$ on the carrier set
of~$\mathbb{A}$ satisfying $\Fn[1]{F} = M$ has automatic action
compatibility \inbr{and thus automatic homeomorphicity} with respect to the
class~$\mathcal{C}'$ of all polymorphism clones of countable
\nbdd{\aleph_0}categorical structures with the strong small index
property and no algebraicity.
Moreover, $M$ has automatic action compatibility with respect
to~$\mathcal{K}'$.
\end{corollary}
\begin{proof}
Except for replacing~$\mathcal{K}$ and~$\mathcal{C}$
by~$\mathcal{K}'$ and~$\mathcal{C}'$, respectively, the proof is
literally identical to the one of Corollary~\ref{cor:combining-results}
with only small changes occurring in the last paragraph:
\nbdd{\forall\exists}interpretations for~$\mathbb{A}$ are not needed
any more since Theorem~\ref{thm:Rubin-group-categoricity} is replaced by
Theorem~\ref{thm:reconstruction-ssip}, and~$G'$ is now the automorphism
group of a countable \nbdd{\aleph_0}categorical structure
without algebraicity whose group, that is~$G'$, has the strong small
index property by the choice of~$\mathcal{K}'$.
\end{proof}

\begin{corollary}\label{cor:reconstruction-ssip}
If $\mathbb{A}=(\Q,\rho)$, where $\rho\in\set{\crc,\sep}$, then
$\End(\mathbb{A})$ and every closed clone~$F$ on the carrier
of~$\mathbb{A}$ with~$\Fn[1]{F}=\End{\mathbb{A}}$ have automatic action
compatibility with respect to countable \nbdd{\aleph_0}categorical
structures with the strong small index property and no algebraicity.
In particular, this holds for the polymorphism clones
$F=\Pol(\mathbb{A})$.
\end{corollary}
\begin{proof}
By Lemma~\ref{lem:str-partially-satisfying-assumptions} all assumptions of
Corollary~\ref{cor:combining-results-ssip} have already been verified
except for the strong small index property.
For $(\Q,\crc)$ this is a direct consequence
(cf.~\cite[Theorem~4.2.9, p.~146]{HodgesModelTheory}) of the
intersection condition proved in Lemma~3.8 of~\cite{TrussVargasReconstructingTopologyMonoidsPolClonesOfRationals},
and is discussed in close proximity to this lemma in the mentioned
article. For $(\Q,\sep)$ the strong small index property has been
observed in~\cite{TrussVargasReconstructingTopologyMonoidsPolClonesOfRationals} directly after Theorem~4.1.
\end{proof}

For $(\Q,\betw)$ the previously presented approach is not applicable
since this structure fails to have the strong small index property. To
see this take for instance $H = \Aut(\Q,<)$, which has index two in
$G=\Aut(\Q,\betw)$. If $H$ were a subset of the setwise stabilizer of a
finite $B\subs \Q$ under $G$, then every order preserving permutation of
$\Q$ would have to preserve~$B$, but unless $B=\emptyset$, this is
impossible. However, for $B=\emptyset$, the pointwise (and setwise)
stabilizer of~$B$ is the whole of~$G$, which is not contained in~$H$.

\paragraph{Acknowledgements}
The authors are highly grateful to Christian Pech for enlightening
discussions on the subject and many valuable comments. In particular,
Christian Pech observed that the argument given in
Corollary~\ref{cor:easier-assumptions} was
already sufficient to cover four more examples that were initially
listed as open. This in turn has led to significant simplification and
streamlining of several proofs in section~\ref{sect:examples}.
Moreover, Christian Pech pointed out that~$(\Q,\betw)$ can be
dealt with using the results from~\cite{BarbinaThesis}. This
initiated the work on Lemmas~\ref{lem:local-generics} and~\ref{lem:Q-circ}
concerning~$(\Q,\crc)$ and~$(\Q,\sep)$. The authors are
also indebted to John Truss for helpful remarks, some clarifications
regarding the four reducts of the rationals treated in the article, and
for moral support of their work. Furthermore, the first
named author thanks Manuel Bodirsky for pointing him to the work of
Paolini and Shelah.


\end{document}